\documentclass[12pt,reqno]{amsart}
\usepackage{amsmath,amssymb,extarrows}
\usepackage{url}
\usepackage{tikz,enumerate}
\usepackage{diagbox}
\usepackage{appendix}
\usepackage{epic}

\usepackage{tcolorbox}

\usepackage{graphicx}
\usepackage{hyperref}
\usepackage{listings}
\usepackage{mathtools}
\usepackage[utf8]{inputenc}

\usepackage{float} 

\usepackage{cite}

\usepackage{hyperref}
\usepackage{array}

\usepackage{booktabs}

\allowdisplaybreaks[4]

\newtheorem{theorem}{Theorem}

\newtheorem{conj}[theorem]{Conjecture}
\newtheorem{lemma}[theorem]{Lemma}

\theoremstyle{definition}

\theoremstyle{remark}
\newtheorem{rem}{Remark}
\numberwithin{equation}{section}
\numberwithin{theorem}{section}
\numberwithin{defn}{section}

\newcommand{\padedvphantom}[3]{%
	\vtop{%
		\vbox{%
			\vspace*{#2}%
			\hbox{\vphantom{#1}}%
		}%
		\vspace*{#3}%
	}%
}

\begin{document}
\title[Zagier's rank two examples for Nahm's problem]
 {Identities on Zagier's rank two examples for Nahm's problem}

\author{Liuquan Wang}
\address{School of Mathematics and Statistics, Wuhan University, Wuhan 430072, Hubei, People's Republic of China}
\email{wanglq@whu.edu.cn;mathlqwang@163.com}




\begin{abstract}
Let $r\geq 1$ be a positive integer, $A$ a real positive definite symmetric $r\times r$ matrix, $B$ a vector of length $r$, and $C$ a scalar. Nahm's problem is to describe all such $A,B$ and $C$ with rational entries for which a specific $r$-fold $q$-hypergeometric series (denoted by $f_{A,B,C}(q)$) involving the parameters $A,B,C$ is modular. When the rank $r=2$, Zagier provided eleven sets of examples of $(A,B,C)$ for which $f_{A,B,C}(q)$ is likely to be modular. We present a number of Rogers--Ramanujan type identities involving double sums, which give modular representations for Zagier's rank two examples. Together with several known cases in the literature, we verified ten of Zagier's examples and give conjectural identities for the remaining example.
\end{abstract}

\maketitle

\section{Introduction}\label{sec-intro}
The famous Rogers--Ramanujan identities, first discovered by Rogers and later rediscovered by Ramanujan, assert that
\begin{align}
\sum_{n=0}^\infty \frac{q^{n^2}}{(q;q)_n}&=\frac{1}{(q,q^4;q^5)_\infty}, \label{RR-1} \\
\sum_{n=0}^\infty \frac{q^{n(n+1)}}{(q;q)_n}&=\frac{1}{(q^2,q^3;q^5)_\infty}. \label{RR-2}
\end{align}
Here and throughout this paper, we assume that $|q|<1$ for convergence and use the standard $q$-series notation
\begin{align}
(a;q)_0:=1, \quad (a;q)_n:=\prod\limits_{k=0}^{n-1}(1-aq^k), \quad (a;q)_\infty :=\prod\limits_{k=0}^\infty (1-aq^k),  \\
(a_1,\cdots,a_m;q)_n:=(a_1;q)_n\cdots (a_m;q)_n, \quad n\in \mathbb{N}\cup \{\infty\}.
\end{align}


Since the appearance of the identities \eqref{RR-1} and \eqref{RR-2}, numerous works have been done to find similar identities, which were usually called as Rogers--Ramanujan type identities. One of the famous works on this topic is Slater's list \cite{Slater}, which contains 130 of such identities such as \cite[Eqs.\ (59),(60),(61)]{Slater}
\begin{align}
\sum_{n=0}^\infty \frac{q^{n^2+2n}(-q;q)_n}{(q;q)_{2n+1}}&=\frac{(q^2,q^{12},q^{14};q^{14})_\infty}{(q;q)_\infty}, \quad \text{(S.\ 59)} \label{Slater59} \\
\sum_{n=0}^\infty \frac{q^{n^2+n}(-q;q)_n}{(q;q)_{2n+1}}&=\frac{(q^4,q^{10},q^{14};q^{14})_\infty}{(q;q)_\infty}, \quad \text{(S.\ 60)} \label{Slater60} \\
\sum_{n=0}^\infty \frac{q^{n^2}(-q;q)_n}{(q;q)_{2n}}&=\frac{(q^6,q^8,q^{14};q^{14})_\infty}{(q;q)_\infty}. \quad \text{(S.\ 61)} \label{Slater61}
\end{align}
Here we use the label S.\ $n$ to denote the equation ($n$) in Slater's list \cite{Slater}. A detailed introduction to Rogers--Ramanujan type identities can be found in Sills' book \cite{Sills-book}.

If we look at \eqref{RR-1} and \eqref{RR-2} from the perspective of modular forms, after multiplying some rational powers of $q$, the product sides of them are modular functions. This fact is not clearly seen from the sum sides. It is thus natural to ask when does a basic hypergeometric series become a modular form. This is an important problem in both the theory of $q$-series and modular forms, and it has not been well understood yet. In particular, let $r\geq 1$ be a positive integer, $A$ a real positive definite symmetric $r\times r$ matrix, $B$ a vector of length $r$, and $C$ a scalar. Let us restrict our attention to the series
$$f_{A,B,C}(q):=\sum_{n=(n_1,\dots,n_r)^\mathrm{T} \in (\mathbb{Z}_{\geq 0})^r} \frac{q^{\frac{1}{2}n^\mathrm{T} An+n^\mathrm{T} B+C}}{(q;q)_{n_1}\cdots (q;q)_{n_r}}.$$
Nahm \cite{Nahm1994,Nahmconf,Nahm2007} posed the following problem: describe all such $A,B$ and $C$ with rational entries for which $f_{A,B,C}(q)$ is a modular function. For convenience, we shall call such $(A,B,C)$ as a modular triple, and call $A$ the matrix part, $B$ the vector part and $C$ the scalar part of it.

The Rogers--Ramanujan identities provided two examples $(A,B,C)=(2,0,-1/60)$ and $(2,1,11/60)$ for Nahm's problem.  Nahm's conjecture as stated in \cite{Zagier} provides a criterion on the matrix $A$ such that there exist $B$ and $C$ with $f_{A,B,C}(q)$ being modular. The conjecture is formulated in terms of Bloch group and a system of polynomial equations induced by $A$. See \cite[p.\ 43]{Zagier} for detailed statement. The motivation of Nahm's  problem comes from physics and the modular functions $f_{A,B,C}(q)$ are expected to be characters of rational conformal field theories.

Zagier \cite{Zagier} studied Nahm's problem and found many possible modular triples. In particular, when the rank $r=1$, Zagier confirmed Nahm's conjecture and proved that there are exactly seven modular triples $(A,B,C)$:
\begin{align}
&\left(1/2,0,-1/40\right), \quad \left(1/2, 1/2, 1/40\right), \quad \left(1,0,-1/48\right), \quad \left(1,1/2,1/24\right), \nonumber \\
&\left(1,-1/2,1/24\right), \quad \left(2,0,-1/60\right), \quad \left(2,1,11/60\right).
\end{align}
When the rank $r\geq 2$, Nahm's problem becomes more difficult. For $r=2$, extensive computer searches have been carried out
by Terhoeven \cite{Terhoeven} and Zagier \cite{Zagier}. After searching over $A=\frac{1}{m}\left(\begin{smallmatrix} a & b \\ b & c\end{smallmatrix}\right)$ with integers $a,b,c,m\leq 100$ which satisfy certain requirements, Zagier found eleven sets of possible modular triples and record them as \cite[Table 2]{Zagier}. To be specific, for $A$ being
\begin{align*}
\begin{pmatrix} \alpha & 1-\alpha \\ 1-\alpha & \alpha \end{pmatrix}, \quad \begin{pmatrix} 2 & 1 \\ 1 & 1 \end{pmatrix}, \quad
\begin{pmatrix} 4 & 1 \\ 1 & 1 \end{pmatrix},  \quad \begin{pmatrix} 4 & 2 \\ 2 & 2 \end{pmatrix}, \quad \begin{pmatrix} 2 & 1 \\ 1 & 3/2 \end{pmatrix}, \quad \begin{pmatrix}  4/3 & 2/3 \\ 2/3 & 4/3 \end{pmatrix}
\end{align*}
and their inverses (here $\alpha\in \mathbb{Q}$ and $\alpha>\frac{1}{2}$), Zagier found several values of $B$ and $C$ for which the function $f_{A,B,C}(q)$ is (or appears to be) modular.  Zagier stated explicit identities which reveal the modularity of $f_{A,B,C}(q)$ only in the case $A=\left(\begin{smallmatrix} \alpha & 1-\alpha \\  1-\alpha & \alpha \end{smallmatrix}\right)$ and $A=\left(\begin{smallmatrix} 4 & 1 \\ 1 & 1 \end{smallmatrix}\right)$. Namely, he proved that \cite[Eq.\ (26)]{Zagier}
\begin{align}\label{Zagier-infinite-case}
f_{\left(\begin{smallmatrix}  \alpha & 1-\alpha \\ 1-\alpha & \alpha\end{smallmatrix}\right), \left(\begin{smallmatrix} \alpha \nu \\ -\alpha\nu \end{smallmatrix}\right), \frac{\alpha}{2}\nu^2-\frac{1}{24}}(q)=\frac{1}{(q;q)_\infty}\sum_{n\in \mathbb{Z}+\nu}q^{\alpha n^2/2-1/24} \quad (\forall \nu \in \mathbb{Q}).
\end{align}
Let $\lfloor x\rfloor $ denote the integer part of $x$. Zagier also stated that \cite[p.\ 45]{Zagier}
\begin{align}
f_{\left(\begin{smallmatrix} 4 & 1 \\ 1 & 1 \end{smallmatrix}\right), \left(\begin{smallmatrix} 0 \\ 1/2\end{smallmatrix}\right), \frac{1}{120}}(q)=\frac{1}{(q;q)_\infty}\sum_{n\equiv 1 \pmod{10}} (-1)^{\lfloor n/10\rfloor}q^{n^2/20-1/24},  \label{Zagier-exam4-1}\\
f_{\left(\begin{smallmatrix} 4 & 1 \\ 1 & 1 \end{smallmatrix}\right), \left(\begin{smallmatrix}2 \\ 1/2 \end{smallmatrix}\right),  \frac{49}{120}}(q)=\frac{1}{(q;q)_\infty}\sum_{n\equiv 3\pmod{10}} (-1)^{\lfloor n/10\rfloor }q^{n^2/20-1/24}, \label{Zagier-exam4-2}
\end{align}
and he remarked that ``these equations were not proved, but only verified to a higher order in the power series in $q$''.

Using an approach outlined by Zagier \cite{Zagier}, Vlasenko and Zwegers \cite{VZ} found all modular triples $(A,B,C)$ for $A$ being $\left(\begin{smallmatrix} a & \lambda-a \\ \lambda-a & a \end{smallmatrix}\right)$ with $a\in \mathbb{Q}$ and $\lambda \in \{\frac{1}{2},1,2\}$. These include several examples missed in Zagier's list \cite[Table 2]{Zagier}. For example, they find that for $A$ being
$\left(\begin{smallmatrix} 3/2 & 1/2 \\ 1/2 & 3/2 \end{smallmatrix}\right)$ or $\left(\begin{smallmatrix} 3/4 & -1/4 \\ -1/4 & 3/4  \end{smallmatrix}\right)$, there are $B$ and $C$ such that $f_{A,B,C}(q)$ is modular. Since these two matrices do not satisfy Nahm's criterion, they can be regarded as counterexamples to Nahm's conjecture as stated in \cite{Zagier}. Vlasenko and Zwegers \cite{VZ} also gave an implicit proof of the formulas \eqref{Zagier-exam4-1}--\eqref{Zagier-exam4-2} (see Remark \ref{rem-exam4}).

Motivated by the above works, the purpose of this paper is to provide a verification of Zagier's list in the rank two case. We will state explicit Rogers--Ramanujan type identities involving double sums for each of Zagier's example. The sum sides of our identities are essentially $f_{A,B,C}(q)$ with $(A,B,C)$ from Zagier's list, and the product sides show clearly that they are indeed modular functions.

For convenience, we label the examples in Zagier's list from 1 to 11 according to their order in \cite[Table 2]{Zagier}. We are able to prove nine of them. Namely, Examples 1--4, 6--9 and 11.  It should be mentioned that seven of Zagier's examples have been discussed (explicitly or implicitly) in the literature. See Table \ref{tab-known} for known cases and references which discuss them. Note that for Example 2, there are five choices for the vector part $B$. Lee's arguments  \cite{LeeThesis} and the work of Vlasenko and Zwegers \cite{VZ} are applicable for three of them. The remaining two cases can be justified by double sum variants of the G\"ollnitz--Gordon identities (see \eqref{Slater34} and \eqref{Slater36}). We will give unified proofs which apply to all these five cases. As for Example 10,  Vlasenko and Zwegers \cite{VZ} only provided conjectural identities, and the modularity of this example were later confirmed by Cherednik and Feigin \cite{Feigin} via the nilpotent double affine Hecke algebras. We will state new identities for Example 10 which is different from but equivalent to that of Vlasenko and Zwegers \cite{VZ}. The only example which remains open is Example 5. We will state conjectural identities for it.

\begin{table}[htbp]
\begin{tabular}{c|cl} \hline
  Exam. No. & Matrix $A$ & References    \\
  \hline
1 & $\left(\begin{smallmatrix}
\alpha & 1-\alpha \\ 1-\alpha &\alpha
\end{smallmatrix} \right)$  &proved by Zagier \cite{Zagier} (see  \eqref{Zagier-infinite-case})  \\
 2 & $\left(\begin{smallmatrix} 2 & 1 \\ 1 & 1 \end{smallmatrix}\right)$ & Lee \cite{LeeThesis}, Vlasenko--Zwegers \cite{VZ} \\
 &  & G\"{o}llnitz--Gordon identities  \\
 3 & $\left(\begin{smallmatrix} 1 & -1 \\ -1 & 2 \end{smallmatrix}\right)$  & proved by Calinescu--Milas--Penn \cite{CMP}  \\
 4 & $\left(\begin{smallmatrix} 4 & 1 \\ 1 & 1 \end{smallmatrix}\right)$ &  proved implicitly by Vlasenko--Zwegers \cite{VZ} \\
 5 & $\left(\begin{smallmatrix}   1/3 & -1/3 \\ -1/3 & 4/3 \end{smallmatrix} \right)$  & Section \ref{sec-exam5} \\
 6 & $\left(\begin{smallmatrix} 4 & 2 \\ 2 & 2 \end{smallmatrix}\right)$ & instances of the Andrews--Gordon identity \cite{LeeThesis}  \\
 7 & $\left(\begin{smallmatrix} 1/2 & -1/2 \\ -1/2 & 1 \end{smallmatrix}  \right)$ & Section \ref{sec-exam7} \\
 8 & $\left( \begin{smallmatrix} 3/2 & 1 \\ 1 & 2 \end{smallmatrix}  \right)$  & Section \ref{sec-exam8} \\
 9 & $\left( \begin{smallmatrix} 1 & -1/2 \\ -1/2 & 3/4 \end{smallmatrix}   \right)$  & Section \ref{sec-exam9} \\
 10 & $\left(\begin{smallmatrix} 4/3 & 2/3 \\ 2/3 & 4/3 \end{smallmatrix}\right)$ & conjectured by Vlasenko--Zwegers \cite{VZ}  \\
 11 & $\left(\begin{smallmatrix} 1 &-1/2\\ -1/2 & 1 \end{smallmatrix}\right)$  &  proved by Vlasenko--Zwegers \cite{VZ} \\
  \hline
\end{tabular}
\\[2mm]
\caption{Zagier's examples and references}
\label{tab-known}
\end{table}

As seen from Table \ref{tab-known}, it appears that this is the first time to state explicit identities for Examples 5, 7, 8 and 9. Moreover, it seems to be the first time to give proofs for Examples 7--9. It is worth mentioning that many of Zagier's examples can be reduced to some known single sum Rogers--Ramanujan type identities from Slater's list \cite{Slater}. The reduction processes vary for different examples. For most of the examples, we achieve it by summing over one of the indexes first. For Example 8 we will use an integral method to find new expressions for the sum sides, and then either eliminate one of the summation indexes or follow the techniques in the author's work \cite{Wang}.

The rest of this paper is organized as follows. In Section \ref{sec-pre} we first collect some useful identities which will play key roles in studying Zagier's examples, and then we introduce a strategy to justify the modularity of generalized Dedekind eta-products. In Section \ref{sec-examples} we discuss Zagier's examples one by one. Since details of proofs for some known examples were omitted in the literature, for the sake of completeness, we shall include complete proofs or brief discussions for all these examples.

\section{Preliminaries}\label{sec-pre}


\subsection{Auxiliary identities}
We first introduce some notations and identities  that will be used frequently.

To make our formulas more compact, sometimes we will use the symbols:
$$J_m:=(q^m;q^m)_\infty, \quad J_{a,m}:=(q^a,q^{m-a},q^m;q^m)_\infty.$$

We need Euler's $q$-exponential identities \cite[Corollary 2.2]{Andrews}
\begin{align}\label{Euler}
\sum_{n=0}^\infty \frac{z^n}{(q;q)_n}=\frac{1}{(z;q)_\infty}, \quad \sum_{n=0}^\infty \frac{q^{\binom{n}{2}} z^n}{(q;q)_n}=(-z;q)_\infty, \quad |z|<1,
\end{align}
and the Jacobi triple product identity \cite[Theorem 2.8]{Andrews}
\begin{align}\label{Jacobi}
(q,z,q/z;q)_\infty=\sum_{n=-\infty}^\infty (-1)^nq^{\binom{n}{2}}z^n.
\end{align}
We also recall Lebesgue's identity \cite[Corollary 2.7]{Andrews}:
\begin{align}\label{Lebesgue}
\sum_{n=0}^\infty \frac{q^{n(n+1)/2}(-zq;q)_n}{(q;q)_n}=(-zq^2;q^2)_\infty (-q;q)_\infty.
\end{align}

As mentioned in the introduction, many of Zagier's examples can be reduced to some identities in Slater's list. Besides \eqref{Slater59}--\eqref{Slater61}, we will also need:
\begin{align}
&\sum_{n=0}^\infty \frac{(-1)^nq^{3n^2}}{(-q;q^2)_n(q^4;q^4)_n} =\frac{(q^2,q^3,q^5;q^5)_\infty}{(q^2;q^2)_\infty}, \quad \text{(S.\ 19)}\label{Slater19} \\
&\sum_{n=0}^\infty \frac{q^{2n(n+1)}(q;q^2)_{n+1}}{(q^2;q^2)_{2n+1}}=\sum_{n=0}^\infty \frac{q^{2n(n+1)}}{(-q;q^2)_{n+1}(q^4;q^4)_n}=\frac{(q,q^6,q^7;q^7)_\infty}{(q^2;q^2)_\infty}, \quad \text{(S.\ 31)} \label{Slater31} \\
&\sum_{n=0}^\infty \frac{q^{2n^2+2n}(q;q^2)_n}{(q^2;q^2)_{2n}}=\sum_{n=0}^\infty \frac{q^{2n(n+1)}}{(q^2;q^2)_n(-q;q)_{2n}}=\frac{(q^2,q^5,q^7;q^7)_\infty}{(q;q)_\infty}, \quad \text{(S.\ 32)}\label{Slater32} \\
&\sum_{n=0}^\infty \frac{q^{2n^2}(q;q^2)_n}{(q^2;q^2)_{2n}}=\sum_{n=0}^\infty \frac{q^{2n^2}}{(q^2;q^2)_n(-q;q)_{2n}}=\frac{(q^3,q^4,q^7;q^7)_\infty}{(q^2;q^2)_\infty}, \quad \text{(S.\ 33)} \label{Slater33} \\
&\sum_{n=0}^\infty \frac{q^{n^2+2n}(-q;q^2)_n}{(q^2;q^2)_n} =\frac{1}{(q^3,q^4,q^5;q^8)_\infty}, \quad \text{(S.\ 34)} \label{Slater34}  \\
&\sum_{n=0}^\infty \frac{q^{n^2}(-q;q^2)_n}{(q^2;q^2)_n}=\frac{1}{(q,q^4,q^7;q^8)_\infty},  \quad \text{(S.\ 36)} \label{Slater36} \\
&\sum_{n=0}^\infty \frac{q^{2n^2}}{(q;q)_{2n}}=\frac{(-q^3,-q^5,q^8;q^8)_\infty}{(q^2;q^2)_\infty}, \quad   \text{(S.\ 39)}\label{Slater39} \\
&\sum_{n=0}^\infty \frac{q^{2n^2+2n}}{(q;q)_{2n+1}}=\frac{(-q,-q^7,q^8;q^8)_\infty}{(q^2;q^2)_\infty}, \quad \text{(S.\ 38)}\label{Slater38} \\
&\sum_{n=0}^\infty \frac{q^{\frac{3}{2}n^2+\frac{3}{2}n}}{(q;q)_n(q;q^2)_{n+1}} =\frac{(q^2,q^8,q^{10};q^{10})_\infty}{(q;q)_\infty}, \quad \text{(S.\ 44)} \label{Slater44} \\
&\sum_{n=0}^\infty \frac{q^{\frac{3}{2}n^2-\frac{1}{2}n}}{(q;q)_n(q;q^2)_n} =\frac{(q^4,q^6,q^{10};q^{10})_\infty}{(q;q)_\infty},   \quad \text{(S.\ 46)} \label{Slater46} \\
&\sum_{n=0}^\infty \frac{q^{(n^2+n)/2}(-q;q)_n}{(q;q)_{2n+1}}=\sum_{n=0}^\infty \frac{q^{(n^2+n)/2}}{(q;q)_n(q;q^2)_{n+1}}=\frac{J_2J_{14}^3}{J_1J_{1,14}J_{4,14}J_{6,14}},  \quad  \text{(S.\ 80)}   \label{Slater80} \\
&\sum_{n=0}^\infty \frac{q^{(n^2+n)/2}(-q;q)_n}{(q;q)_{2n}}
=\sum_{n=0}^\infty \frac{q^{(n^2+n)/2}}{(q;q)_n(q;q^2)_n}=\frac{J_2J_{14}^3}{J_1J_{2,14}J_{3,14}J_{4,14}}, \quad   \text{(S.\ 81)}  \label{Slater81} \\
&\sum_{n=0}^\infty \frac{q^{(n^2+3n)/2}(-q;q)_n}{(q;q)_{2n+1}}=\sum_{n=0}^\infty \frac{q^{(n^2+3n)/2}}{(q;q)_n(q;q^2)_{n+1}}=\frac{J_2J_{14}^3}{J_1J_{2,14}J_{5,14}J_{6,14}}, \quad  \text{(S.\ 82)}\label{Slater82} \\
&\sum_{n=0}^\infty \frac{q^{3n^2+2n}(-q;q^2)_{n+1}}{(q^2;q^2)_{2n+1}}=\frac{J_2J_{3,10}J_{4,20}}{J_1J_4J_{20}}, \quad \text{(S.\ 97)} \label{Slater97} \\
&\sum_{n=0}^\infty \frac{q^{n^2}(-q;q^2)_n}{(q^2;q^2)_{2n}}=\sum_{n=0}^\infty \frac{q^{n^2}}{(q;q^2)_n(q^4;q^4)_n}=\frac{J_2J_{14}J_{3,28}J_{11,28}}{J_1J_{28}J_{4,28}J_{12,28}}, \quad  \text{(S.\ 117)} \label{Slater117} \\
&\sum_{n=0}^\infty \frac{q^{n^2+2n}(-q;q^2)_n}{(q^2;q^2)_{2n}}=\sum_{n=0}^\infty \frac{q^{n^2+2n}}{(q;q^2)_n(q^4;q^4)_n}=\frac{J_2J_{1,14}J_{12,28}}{J_1J_4J_{28}}, \quad \text{(S.\ 118)} \label{Slater118} \\
&\sum_{n=0}^\infty \frac{q^{n^2+2n}(-q;q^2)_{n+1}}{(q^2;q^2)_{2n+1}}=\sum_{n=0}^\infty \frac{q^{n^2+2n}}{(q;q^2)_{n+1}(q^4;q^4)_n}=\frac{J_2J_{4,28}J_{5,14}}{J_1J_4J_{28}}. \quad  \text{(S.\ 119)} \label{Slater119}
\end{align}
The identities \eqref{Slater31}--\eqref{Slater33} are usually referred as the Rogers--Selberg mod 7 identities.
Note that \eqref{Slater34} also appeared as \cite[Entry 1.7.12]{Lost2}, \eqref{Slater36} also appeared in \cite[Entries 1.7.11 and 4.2.15]{Lost2}. A typo has been corrected for \eqref{Slater97}. The identities \eqref{Slater34} and \eqref{Slater36} are usually referred as the G\"{o}llnitz--Gordon identities.

For Example 8 we will provide two different proofs. For both proofs we will use an integral method to rewrite the sum sides as integrals of some infinite products. This method was applied by Rosengren \cite{Rosengren} to prove some conjectural identities of Kanade and Russell \cite{KR-2019}. Later it has been applied in several works. For example, it was used by Chern \cite{Chern} and the author \cite{Wang} to prove a conjecture of Andrews and Uncu \cite{Andrews-Uncu}. It was also utilized by Mc Laughlin \cite{Laughlin} and Cao and Wang \cite{Cao-Wang} in finding some new multi-sum Rogers--Ramanujan type identities.

In particular, for the second proof of Example 8, we rely on the following result found from the book of Gasper and Rahman \cite{GR-book}, which plays a key role in the author's work \cite{Wang}. Before stating it, we remark that the symbol ``idem $(c_1;c_2,\dots,c_C)$'' after an expression stands for the sum of the $(C-1)$ expressions obtained from the preceding expression by interchanging $c_1$ with each $c_k$, $k=2,3,\dots,C$.

\begin{lemma}\label{lem-integral}
(Cf.\ \cite[Eq.\ (4.10.6)]{GR-book})
Suppose that
$$P(z):=\frac{(a_1z,\dots,a_Az,b_1/z,\dots,b_B/z;q)_\infty}{(c_1z,\dots,c_Cz,d_1/z,\dots,d_D/z;q)_\infty}$$
has only simple poles. We have
\begin{align}\label{eq-integral}
\oint P(z)\frac{\mathrm{d}z}{2\pi iz}=& \frac{(b_1c_1,\dots,b_Bc_1,a_1/c_1,\dots,a_A/c_1;q)_\infty }{(q,d_1c_1,\dots,d_Dc_1,c_2/c_1,\dots,c_C/c_1;q)_\infty} \nonumber \\
& \times \sum_{n=0}^\infty \frac{(d_1c_1,\dots,d_Dc_1,qc_1/a_1,\dots,qc_1/a_A;q)_n}{(q,b_1c_1,\dots,b_Bc_1,qc_1/c_2,\dots,qc_1/c_C;q)_n} \nonumber \\
&\times \Big(-c_1q^{(n+1)/2}\Big)^{n(C-A)}\Big(\frac{a_1\cdots a_A}{c_1\cdots c_C} \Big)^n +\text{idem} ~(c_1;c_2,\dots,c_C)
\end{align}
when $C>A$, or if $C=A$ and
\begin{align}\label{cond}
\left|\frac{a_1\cdots a_A }{c_1\cdots c_C}\right|<1.
\end{align}
Here the integration is over a positively oriented contour so that the poles of
$$(c_1z,\dots,c_Cz;q)_\infty^{-1}$$
lie outside the contour, and the origin and poles of $(d_1/z,\dots,d_D/z;q)_\infty^{-1}$ lie inside the contour.
\end{lemma}

\subsection{Basic facts about modular forms}\label{sec-modularity}
We will need some basic facts about modular forms, which can be found in the work of Frye and Garvan \cite{Garvan-Liang} and the references therein. We always assume that $q=e^{2\pi i\tau}$ where  $\tau \in \mathbb{H}:=\{\tau\in \mathbb{C}: \mathrm{Im} \tau>0\}$. The Dedekind eta function is defined by
\begin{align}\label{eq-eta-defn}
\eta(\tau):=q^{1/24}\prod\limits_{n=1}^\infty (1-q^n).
\end{align}
The generalized Dedekind eta function is defined to be
\begin{align}\label{eq-general-eta}
\eta_{\delta;g}(\tau):=q^{\frac{\delta}{2}P_2(g/\delta)}\prod\limits_{m\equiv \pm g \pmod{\delta}} (1-q^m),
\end{align}
where $P_2(t)=\{t\}^2-\{t\}+\frac{1}{6}$ is the second periodic Bernoulli polynomial, $\{t\}=t-\lfloor t\rfloor $ is the fractional part of $t$, $g,\delta,m\in \mathbb{Z}^+$ and $0<g<\delta$. We note the following obvious but useful fact: for any positive integer $k$, we have
\begin{align}\label{eta-scaling}
\eta_{k\delta;kg}(\tau)=\eta_{\delta;g}(k\tau).
\end{align}

Recall the full modular group
\begin{align*}
\mathrm{SL}_2(\mathbb{Z}):=\left\{\begin{pmatrix} a & b \\ c & d \end{pmatrix}: ad-bc=1,a,b,c,d\in \mathbb{Z}\right\}
\end{align*}
and one of its congruence subgroup
\begin{align*}
\Gamma_1(N):=\left\{\gamma \in \mathrm{SL}_2(\mathbb{Z}):\gamma \equiv \begin{pmatrix} 1 & * \\ 0 & 1 \end{pmatrix} \pmod{N} \right\}.
\end{align*}
As is well known, the function $\eta_{\delta;g}(\tau)$ is a modular function on some congruence subgroup of $\mathrm{SL}_2(\mathbb{Z})$ with multiplier system. Let $N$ be a fixed positive integer. A generalized Dedekind eta-product of level $N$ has the form
\begin{align}\label{defn-eq-general-eta}
f(\tau)=\prod\limits_{\begin{smallmatrix} \delta|N \\ 0<g<\delta \end{smallmatrix}} \eta_{\delta;g}^{r_{\delta,g}}(\tau)
\end{align}
where
\begin{align}
r_{\delta,g}\in \left\{\begin{array}{ll}
\frac{1}{2}\mathbb{Z} & \text{if $g=\delta/2$,}\\
\mathbb{Z} & \text{otherwise.}
\end{array}\right.
\end{align}
Recall the following result of Robins \cite{Robins}.
\begin{lemma}\label{lem-Robins}
The function $f(\tau)$ defined in \eqref{defn-eq-general-eta} is a modular function on $\Gamma_1(N)$ if
\begin{enumerate}
\item $\sum\limits_{\begin{smallmatrix} \delta|N \\ g\end{smallmatrix}} \delta P_2(\frac{g}{\delta})r_{\delta,g}\equiv 0 \pmod{2}$ and
\item  $\sum\limits_{\begin{smallmatrix} \delta|N \\ g \end{smallmatrix}} \frac{N}{\delta} P_2(0)r_{\delta,g}\equiv 0 \pmod{2}$.
\end{enumerate}
\end{lemma}

Given a generalized eta-product $f(\tau)$ as in \eqref{defn-eq-general-eta}, suppose that
\begin{align}
\sum\limits_{\begin{smallmatrix} \delta|N \\ g\end{smallmatrix}} \delta P_2(\frac{g}{\delta})r_{\delta,g}=\frac{m_1}{n_1}, \quad
\sum\limits_{\begin{smallmatrix} \delta|N \\ g \end{smallmatrix}} \frac{N}{\delta} P_2(0)r_{\delta,g}=\frac{m_2}{n_2}
\end{align}
where $m_i,n_i\in \mathbb{Z}$, $n_i>0$ and $\gcd(m_i,n_i)=1$ for $i=1,2$. We can find the least positive integers $k$ and $N_0$ such that both $km_1/n_1$ and $N_0m_2/n_2$ are even integers. Then in view of \eqref{eta-scaling} and Lemma \ref{lem-Robins}, we know that $f(k\tau)$ is a modular function on $\Gamma_1(kN_0N)$.

Next, suppose that $f_i(\tau)$ is a generalized eta-product ($i=1,2,\dots,s$). It is clear that
$$F(\tau):=f_1(\tau)+f_2(\tau)+\cdots+f_s(\tau)$$
is a modular function. To make this more precise, let $k_i$ be a positive integer such that $f_i(k_i\tau)$ is a modular function on $\Gamma_1(N_i)$ ($i=1,2,\dots,s$). Let $k$ (resp.\ $N$) be the least common multiple of $k_1,k_2,\dots,k_s$ (resp.\ $N_1,N_2,\dots,N_s$), then it is easy to see that $F(k\tau)$ is a modular function on $\Gamma_1(kN)$.

In our situation, we will convert a Nahm sum $f_{A,B,0}(q)$ to sum of infinite products:
$$f_{A,B,0}(q)=\lambda_1q^{c_1}g_1(q)+\lambda_2 q^{c_2}g_2(q)+\cdots +\lambda_s q^{c_s}g_s(q)$$
where for $i=1,2,\dots,s$, $\lambda_i\in \mathbb{C}$ and each $g_i(q)$ is a product consists of the functions $J_m$ and $J_{a,m}$. If there exists a common rational number $C$ such that each $q^{C+c_i}g_i(q)$ ($i=1,2,\dots,s$) becomes a generalized eta-product, then $f_{A,B,C}(q)=q^Cf_{A,B,0}(q)$ is a modular function.
The value $C$ could be easily computed using the definition \eqref{defn-eq-general-eta} with the help of the Maple package \texttt{thetaids} developed by Frye and Garvan \cite{Garvan-Liang}. This package can be used to justify the modularity of generalized eta-products and prove identities among them. Here we only give a brief explanation of its application to justify modularity.

The generalized eta-product
$$\eta_{N;g_1}^{r_1}(\tau)\eta_{N;g_2}^{r_2}(\tau)\cdots \eta_{N;g_m}^{r_m}(\tau)$$
is symbolically represented by the so-called geta-list
$$[[N, g_1,r_1],[N, g_2,r_2] \ldots, [N, g_m,r_m]].$$
The command \texttt{GETAP2getalist} converts a product of generalized eta-functions into a list as described above.

The functions on the left side of (1) and (2) in Lemma \ref{lem-Robins} are computed using the Maple functions \texttt{vinf} and \texttt{v0} respectively. Based on Robins' result, the command \texttt{Gamma1ModFunc(L,N)} checks whether a given generalized eta-product is a modular function on $\Gamma_1(N)$, where the generalized et-product is encoded as the geta-list $L$. It returns $1$ if it is a modular function on $\Gamma_1(N)$ otherwise it returns $0$.  If the global variable \texttt{xprint} is set to true then the verification process is printed. We will demonstrate this process briefly in Example 7.

\section{Explicit identities for rank two examples}\label{sec-examples}

Now we discuss Zagier's examples one by one. The modular triples for Examples 1-11 \cite[Table 2]{Zagier} are recorded as Tables \ref{tab-1}--\ref{tab-11}, respectively. We will also determine a positive integer $k$ and the level $N$ such that $f_{A,B,C}(q^k)$ is a modular function on $\Gamma_1(N)$. We stress here that the symbols $F(u,v;q)$, $F_i(q), G_i(q), H_i(q), R_i(q), S_i(q)$ and  $T_i(q)$ used below may have different meanings in each proof. Their precise meanings will be given in the context.

\subsection{Example 1}\label{sec-exam1}
The modular triple $(A,B,C)$ for this example is given in Table \ref{tab-1}.
\begin{table}[H]
\centering
\begin{tabular}{c|c}
  \hline
  \padedvphantom{I}{3ex}{3ex}
  $A$ & $\begin{pmatrix}
\alpha & 1-\alpha \\ 1-\alpha &\alpha
\end{pmatrix}$  \\
\hline
\padedvphantom{I}{3ex}{3ex}
  $B$ & $\begin{pmatrix} \alpha \nu \\ -\alpha \nu \end{pmatrix}$, \quad $\nu \in \mathbb{Q}$  \\
~~$C$ & $\alpha \nu^2/2-1/24$ \\
  \hline
\end{tabular}
\\[2mm]
\caption{Modular triples for Example 1}\label{tab-1}
\end{table}
Here to make $A$ a positive definite rational matrix we need $\alpha\in \mathbb{Q}$ and $\alpha>\frac{1}{2}$. Zagier stated \eqref{Zagier-infinite-case} and provided a proof for it. Utilizing \eqref{Jacobi}, his result can be stated in the following equivalent form.
\begin{theorem}\label{thm-1}
We have for $\alpha>\frac{1}{2}$ that
\begin{align}
\sum_{i,j\geq 0}\frac{q^{\frac{\alpha}{2}i^2+(1-\alpha)ij+\frac{\alpha}{2}j^2+\alpha \nu i -\alpha \nu j}}{(q;q)_i(q;q)_j}=\frac{(-q^{\frac{\alpha}{2}+\alpha \nu},-q^{\frac{\alpha}{2}-\alpha \nu},q^{\alpha};q^{\alpha})_\infty}{(q;q)_\infty}.
\end{align}
\end{theorem}
For specific values of $\alpha$ and $\nu$, we can find the level $N$ for which the Nahm sum in this example is essentially modular on $\Gamma_1(N)$.

Here we reproduce Zagier's proof \cite[p.\ 46]{Zagier}.
\begin{proof}
We have
\begin{align*}
\text{LHS}&=\sum_{i,j\geq 0}\frac{q^{\frac{\alpha}{2}(i-j)^2+\alpha \nu(i-j)+ij}}{(q;q)_i(q;q)_j} =\sum_{n=-\infty}^\infty q^{\frac{\alpha}{2}n^2+\alpha \nu n} \sum_{i-j=n}\frac{q^{ij}}{(q;q)_i(q;q)_j} \\
&=\frac{1}{(q;q)_\infty}\sum_{n=-\infty}^\infty q^{\frac{\alpha}{2}n^2+\alpha \nu n}=\text{RHS}.
\end{align*}
Here for the last second equality we used the Durfee rectangle identity: for any fixed integer $n$,
\begin{align}\label{id-diff}
\sum_{i-j=n}  \frac{q^{ij}}{(q;q)_i(q;q)_j}=\sum_{j=0}^\infty \frac{q^{j(j+n)}}{(q;q)_j(q;q)_{j+n}}=\frac{1}{(q;q)_\infty},
\end{align}
and for the last equality we employed \eqref{Jacobi}.
\end{proof}

Theorem \ref{thm-1} can also be regarded as a special case of the following identity found by Cao and Wang \cite[Theorem 3.4]{Cao-Wang}:
\begin{align}\label{eq-Cao-Wang-1}
\sum_{i,j\geq 0} \frac{u^{i-j}q^{\binom{i}{2}+\binom{j+1}{2}+a\binom{j-i}{2}}}{(q;q)_i(q;q)_j}=\frac{(-uq^a,-q/u,q^{a+1};q^{a+1})_\infty}{(q;q)_\infty}.
\end{align}
In fact, setting $a=\alpha-1$ and $u=q^{\alpha \nu+1-\frac{\alpha}{2}}$ in \eqref{eq-Cao-Wang-1}, we recover Theorem \ref{thm-1}. Note that the proof of the identity \eqref{eq-Cao-Wang-1} given in \cite{Cao-Wang} is different from the above. It does not use \eqref{id-diff}. Instead, it is based on the integral method and repeated use of \eqref{Jacobi}.

%

\subsection{Example 2.}\label{sec-exam2}
The modular triples for this example are given in Table \ref{tab-2}.
\begin{table}[H]
\centering
\begin{tabular}{c|ccccc}
\hline
\padedvphantom{I}{3ex}{3ex}
  $A$ & \multicolumn{5}{c}{$\begin{pmatrix} 2 & 1 \\ 1 & 1 \end{pmatrix}$}
   \\
  \hline
    \padedvphantom{I}{3ex}{3ex}
  $B$ & $\begin{pmatrix} -1 \\ 1/2 \end{pmatrix}$ & $\begin{pmatrix} 0 \\ 0 \end{pmatrix}$  & $\begin{pmatrix} 0 \\ 1/2 \end{pmatrix}$
& $\begin{pmatrix} 1 \\ 1/2 \end{pmatrix}$ & $\begin{pmatrix} 1 \\ 1  \end{pmatrix}$    \\
  ~~$C$ & $1/8$ & $-1/32$ & $0$ & $1/8$ & $7/32$   \\
  ~~$N$ & $256$ & $1024$ & $32$ & $256$ & $1024$ \\
  \hline
\end{tabular}
\\[2mm]
\caption{Modular triples for Example 2}\label{tab-2}
\end{table}
Let $\mathrm{denom}(C)$ denote the denominator of a rational number $C$ in reduced form. Here and in Table \ref{tab-3} below, we record in the fourth row the levels $N$ such that the Nahm sums $f_{A,B,C}(q^{\mathrm{denom}(C)})$ are  modular functions on $\Gamma_1(N)$.

\begin{theorem}\label{thm-2}
We have
\begin{align}
\sum_{i,j\geq 0} \frac{q^{i^2+ij+\frac{1}{2}j^2-i+\frac{1}{2}j}}{(q;q)_i(q;q)_j}&=2\frac{(q^4;q^4)_\infty}{(q;q)_\infty},  \label{exam2-1}\\
\sum_{i,j\geq 0} \frac{q^{2i^2+2ij+j^2}}{(q^2;q^2)_i(q^2;q^2)_j}&=\frac{1}{(q,q^4,q^7;q^8)_\infty}, \label{exam2-2}\\
\sum_{i,j\geq 0}\frac{q^{i^2+ij+\frac{1}{2}j^2+\frac{1}{2}j}}{(q;q)_i(q;q)_j}&=\frac{(q^2;q^2)_\infty^3}{(q;q)_\infty^2 (q^4;q^4)_\infty}, \label{exam2-3} \\
\sum_{i,j\geq 0}\frac{q^{i^2+ij+\frac{1}{2}j^2+i+\frac{1}{2}j}}{(q;q)_i(q;q)_j}&=\frac{(q^4;q^4)_\infty}{(q;q)_\infty},  \label{exam2-4}\\
\sum_{i,j\geq 0}\frac{q^{2i^2+2ij+j^2+2i+2j}}{(q^2;q^2)_i(q^2;q^2)_j}&=\frac{1}{(q^3,q^4,q^5;q^8)_\infty}. \label{exam2-5}
\end{align}
As a consequence, the Nahm sums $f_{A,B,C}(q^{\mathrm{denom}(C)})$ are modular functions on $\Gamma_1(N)$ for $(A,B,C,N)$ given in Table \ref{tab-2}.
\end{theorem}
Here for the second and fifth identities, we have replaced $q$ by $q^2$ so that the series contain only integral powers of $q$. We will do similar replacements for other examples as well. Here and in most of the examples, we will only prove the identities. The last assertions on modularity can be easily verified following the discussion in Section \ref{sec-modularity}. See Example 7 for instance of such verifications.
\begin{proof}
We define
\begin{align}
F(u,v;q)&=\sum_{i,j\geq 0}\frac{q^{i^2+ij+\frac{1}{2}j^2}u^iv^j}{(q;q)_i(q;q)_j} =\sum_{i\geq 0}\frac{q^{i^2}u^i}{(q;q)_i} \sum_{j\geq 0}\frac{q^{\frac{1}{2}(j^2-j)}\cdot (q^{i+\frac{1}{2}}v)^j}{(q;q)_j} \nonumber \\
&=\sum_{i\geq 0}\frac{q^{i^2}u^i}{(q;q)_i}(-q^{i+\frac{1}{2}}v;q)_\infty  =(-q^{\frac{1}{2}}v;q)_\infty \sum_{i\geq 0}\frac{q^{i^2}u^i}{(q,-q^{\frac{1}{2}}v;q)_i}. \label{Exam2-F}
\end{align}

Setting $(u,v)=(q^{-1},q^{\frac{1}{2}})$, we have by \eqref{Exam2-F} and \eqref{Euler} that
\begin{align}
F(q^{-1},q^{\frac{1}{2}};q)&=(-q;q)_\infty \sum_{i=0}^\infty \frac{q^{i^2-i}}{(q^2;q^2)_i}=(-q;q)_\infty (-1;q^2)_\infty=2\frac{(q^4;q^4)_\infty}{(q;q)_\infty}.
\end{align}
This proves \eqref{exam2-1}.

Setting $(u,v,q)$ as $(1,1,q^2)$, we have by \eqref{Exam2-F} and \eqref{Euler} that
\begin{align}\label{exam2-2-proof}
F(1,1;q^2)=(-q;q^2)_\infty \sum_{i=0}^\infty \frac{q^{2i^2}}{(-q,q^2;q^2)_i}.
\end{align}
Replacing $q$ by $-q$ in \eqref{Slater39} and substituting it into \eqref{exam2-2-proof}, we obtain \eqref{exam2-2}.

Setting $(u,v)=(1,q^{\frac{1}{2}})$ and using \eqref{Euler}, we have
\begin{align}
F(1,q^{\frac{1}{2}};q)=(-q;q)_\infty \sum_{i=0}^\infty \frac{q^{i^2}}{(q^2;q^2)_i}=(-q;q)_\infty (-q;q^2)_\infty =\frac{(q^2;q^4)_\infty}{(q;q^2)_\infty^2}.
\end{align}
This proves \eqref{exam2-3}.

Setting $(u,v)=(q,q^{\frac{1}{2}})$ and using \eqref{Euler}, we have
\begin{align}
F(q,q^{\frac{1}{2}};q)=(-q;q)_\infty \sum_{i=0}^\infty \frac{q^{i^2+i}}{(q^2;q^2)_i}=(-q;q)_\infty (-q^2;q^2)_\infty =\frac{(q^4;q^4)_\infty}{(q;q)_\infty}.
\end{align}
This proves \eqref{exam2-4}.

Setting $(q,u,v)$ as $(q^2,q^2,q^2)$, we have by \eqref{Exam2-F}  that
\begin{align}
F(q^2,q^2;q^2)&=(-q^3;q^2)_\infty \sum_{i=0}^\infty \frac{q^{2i^2+2i}}{(q^2,-q^3;q^2)_i} =(-q;q^2)_\infty \sum_{i=0}^\infty \frac{q^{2i^2+2i}}{(-q;q^2)_{i+1}(q^2;q^2)_i}. \label{exam2-5-proof}
\end{align}
Replacing $q$ by $-q$ in \eqref{Slater38} and then substituting it into \eqref{exam2-5-proof}, we obtain \eqref{exam2-5}.
\end{proof}
\begin{rem}
This example has been discussed by Lee \cite[Example 3.5.2]{LeeThesis}. He offered a proof of \eqref{exam2-4} using Lebesgue's identity \eqref{Lebesgue}. To be specific, using the $q$-binomial identity
$$(-zq;q)_k=\sum_{r=0}^k \frac{(q;q)_k}{(q;q)_r(q;q)_{k-r}}q^{r(r+1)/2}z^r,$$
Lee \cite[Eq.\ (3.25)]{LeeThesis} rewrote the sum side of \eqref{Lebesgue} as
\begin{align}\label{eq-Lee}
\sum_{k=0}^\infty \frac{q^{k(k+1)/2}(-zq;q)_k}{(q;q)_k}=\sum_{i,j\geq 0}\frac{z^iq^{i^2+ij+\frac{j^2}{2}+i+\frac{j}{2}}}{(q;q)_i(q;q)_j}.
\end{align}
Setting $z=q^{-2}$, $q^{-1}$ and 1 in \eqref{eq-Lee}, we obtain \eqref{exam2-1}, \eqref{exam2-3} and \eqref{exam2-4}, respectively. But one cannot get \eqref{exam2-2} and \eqref{exam2-5} by specializing the choice of $z$. These two identities were also known in the literature. For example, Uncu \cite[Corollary 3.8]{Uncu} gave some double sum representations of the G\"ollnitz--Gordon identities (see \eqref{Slater34} and \eqref{Slater36}), which prove \eqref{exam2-2} and \eqref{exam2-5}.
\end{rem}
\begin{rem}
Warnaar pointed out to us that \eqref{exam2-2} is a special case of the following identity \cite[p.\ 235]{Warnaar}:
\begin{align}\label{eq-Warnaar}
\sum_{n_1,\dots,n_{k-1}\geq 0} \frac{q^{\frac{1}{2}(N_1^2+\cdots+N_{k-1}^2)}}{(q;q)_{n_1}\cdots (q;q)_{n_{k-1}}}=\frac{(-q^{\frac{1}{2}};q)_\infty}{(q;q)_\infty}(q^{\frac{k}{2}},q^{\frac{k}{2}+1},q^{k+1};q^{k+1})_\infty,
\end{align}
where $N_i=n_i+\cdots +n_{k-1}$. Indeed, setting $k=3$ in this identity we obtain \eqref{exam2-2}. Meanwhile, using Lebesgue's identity, we can also prove \eqref{exam2-4} using Lemma A.1 in \cite{Warnaar}.
\end{rem}
\begin{rem}\label{rem-exam2}
Let $A$ be a real positive definite symmetric $r\times r$-matrix, $B$ a vector of length $r$ and $C$ a scalar. Let $A',B'$ and $C'$ be the symmetric $2r\times 2r$-matrix, the vector of length $2r$ and the scalar, resp., given by
$$A'=\begin{pmatrix} 2A &I_r \\ I_r & I_r \end{pmatrix}, \quad B'=\begin{pmatrix} 2B \\ 1/2 \\ \vdots \\ 1/2 \end{pmatrix}, \quad C'=2C+\frac{r}{24}.$$
Vlasenko and Zwegers \cite[Theorem 4.2]{VZ} proved that
\begin{align}\label{eq-VZ}
f_{A',B',C'}(q)=q^{r/24}\frac{(q^2;q^2)_\infty^r}{(q;q)_\infty^r}f_{A,B,C}(q^2).
\end{align}
If we set $r=1$, $(A,B,C)=(1,-1/2,1/24)$, $(1,0,-1/48)$ and $(1,1/2,1/24)$ in \eqref{eq-VZ} and then use \eqref{Euler}, we obtain \eqref{exam2-1}, \eqref{exam2-3} and \eqref{exam2-4}, respectively.
\end{rem}

\subsection{Example 3}\label{sec-exam3}
The modular triples for this example are given in Table \ref{tab-3}\footnote{Here the last value of $C$ has been corrected. In \cite[Table 2]{Zagier} it was written as $19/24$.}.
\begin{table}[H]
\centering
\begin{tabular}{c|ccccc}
  \hline
    \padedvphantom{I}{3ex}{3ex}
  $A$ &  \multicolumn{5}{c}{$\begin{pmatrix} 1 & -1 \\ -1 & 2 \end{pmatrix}$}   \\
  \hline
    \padedvphantom{I}{3ex}{3ex}
  $B$ & $\begin{pmatrix}  -3/2 \\ 2 \end{pmatrix}$ & $\begin{pmatrix} 0 \\ 0  \end{pmatrix}$ & $\begin{pmatrix} -1/2 \\ 1 \end{pmatrix}$ &  $\begin{pmatrix} 1/2 \\ 0  \end{pmatrix}$ & $\begin{pmatrix} 0 \\ 1  \end{pmatrix}$    \\
  ~~$C$ & $25/24$ & $-5/96$ & $1/6$ & $1/24$ & $19/96$ \\
  ~~$N$ & $2304$ & $9216$ & $576$ & $2304$ & $9216$ \\
  \hline
\end{tabular}
\\[2mm]
\caption{Modular triples for Example 3}\label{tab-3}
\end{table}

\begin{theorem}\label{thm-3}
We have
\begin{align}
\sum_{i,j\geq 0}\frac{q^{\frac{1}{2}i^2-ij+j^2-\frac{3}{2}i+2j}}{(q;q)_i(q;q)_j} &=2q^{-1}\frac{(q^2;q^2)_\infty^4}{(q;q)_\infty^3 (q^4;q^4)_\infty}, \label{exam3-1} \\
\sum_{i,j\geq 0}\frac{q^{i^2-2ij+2j^2}}{(q^2;q^2)_i (q^2;q^2)_j} &=\frac{(q^2;q^2)_\infty^3 (q^3,q^5,q^8;q^8)_\infty}{(q;q)_\infty^2 (q^4;q^4)_\infty^2}, \label{exam3-2} \\
\sum_{i,j\geq 0}\frac{q^{\frac{1}{2}i^2-ij+j^2-\frac{1}{2}i+j}}{(q;q)_i (q;q)_j}&=2\frac{(q^2;q^2)_\infty (q^4;q^4)_\infty}{(q;q)_\infty^2}, \label{exam3-3} \\
\sum_{i,j\geq 0}\frac{q^{\frac{1}{2}i^2-ij+j^2+\frac{1}{2}i}}{(q;q)_i (q;q)_j} &=\frac{(q^2;q^2)_\infty^4}{(q;q)_\infty^3 (q^4;q^4)_\infty}, \label{exam3-4} \\
\sum_{i,j\geq 0} \frac{q^{i^2-2ij+2j^2+2j}}{(q^2;q^2)_i (q^2;q^2)_j} &=\frac{(q^2;q^2)_\infty^3  (q,q^7,q^8;q^8)_\infty}{(q;q)_\infty^2 (q^4;q^4)_\infty^2}. \label{exam3-5}
\end{align}
As a consequence, the Nahm sums $f_{A,B,C}(q^{\mathrm{denom}(C)})$ are modular functions on $\Gamma_1(N)$ for $(A,B,C,N)$ given in Table \ref{tab-3}.
\end{theorem}
\begin{proof}
We define
\begin{align}
F(u,v;q):=\sum_{i,j\geq 0}\frac{q^{\frac{1}{2}i^2-ij+j^2}u^iv^j}{(q;q)_i(q;q)_j}.
\end{align}
Summing over $i$ first using \eqref{Euler}, we obtain
\begin{align}\label{exam3-F}
F(u,v;q)=\sum_{j=0}^\infty \frac{q^{j^2}v^j (-uq^{\frac{1}{2}-j};q)_\infty}{(q;q)_j}.
\end{align}

Setting $(u,v)=(q^{-\frac{3}{2}},q^2)$ in \eqref{exam3-F}, we have
\begin{align*}
&F(q^{-\frac{3}{2}},q^2;q)=\sum_{j=0}^\infty \frac{q^{j^2+2j}}{(q;q)_j}(-q^{-1-j};q)_\infty \nonumber \\
&=2q^{-1}(-q;q)_\infty \sum_{j=0}^\infty \frac{q^{j(j+1)/2} (-q;q)_{j+1}}{(q;q)_j} =2q^{-1}(-q;q)_\infty^2 (-q;q^2)_\infty.
\end{align*}
This proves \eqref{exam3-1}. Here the last equality follows from Lebesgue's identity \eqref{Lebesgue} with $z=q$.

Setting $(u,v,q)$ as $(1,1,q^2)$ in \eqref{exam3-F}, we have
\begin{align}
F(1,1;q^2)&=\sum_{j=0}^\infty \frac{q^{2j^2}}{(q^2;q^2)_j}(-q^{1-2j};q^2)_\infty =(-q;q^2)_\infty \sum_{j=0}^\infty \frac{q^{j^2}(-q;q^2)_j}{(q^2;q^2)_j} \nonumber \\
&=(-q;q^2)_\infty^2 \frac{(q^3,q^5,q^8;q^8)_\infty}{(q^2;q^2)_\infty}.
\end{align}
This proves \eqref{exam3-2}. Here for the last equality we used \eqref{Slater36}.

Setting $(u,v)=(q^{-\frac{1}{2}},q)$ in \eqref{exam3-F}, we have
\begin{align}
F(q^{-\frac{1}{2}},q;q)&=\sum_{j=0}^\infty \frac{q^{j^2+j}}{(q;q)_j}(-q^{-j};q)_\infty =(-1;q)_\infty \sum_{j=0}^\infty \frac{q^{j(j+1)/2} (-q;q)_j}{(q;q)_j} \nonumber \\
&=2(-q;q)_\infty^2 (-q^2;q^2)_\infty.
\end{align}
This proves \eqref{exam3-3}. Here for the last equality we used \eqref{Lebesgue} with $z=1$.

Setting $(u,v)=(q^{\frac{1}{2}},1)$ in \eqref{exam3-F}, we have
\begin{align}
F(q^{\frac{1}{2}},1;q)&=\sum_{j=0}^\infty \frac{q^{j^2}(-q^{1-j};q)_\infty}{(q;q)_j} =(-q;q)_\infty \sum_{j=0}^\infty \frac{q^{j(j+1)/2}(-1;q)_j}{(q;q)_j} \nonumber \\
&=(-q;q)_\infty^2 (-q;q^2)_\infty.
\end{align}
This proves \eqref{exam3-4}. Here for the last equality we used \eqref{Lebesgue} with $z=q^{-1}$.

Setting $(u,v,q)$ as $(1,q^2,q^2)$ in \eqref{exam3-F}, we have
\begin{align}
F(1,q^2;q^2)&=\sum_{j=0}^\infty \frac{q^{2j^2+2j}}{(q^2;q^2)_j}(-q^{1-2j};q^2)_\infty =(-q;q^2)_\infty \sum_{j=0}^\infty \frac{q^{j^2+2j}(-q;q^2)_j}{(q^2;q^2)_j}. \label{exam3-final}
\end{align}
Substituting \eqref{Slater34} into \eqref{exam3-final}, we obtain \eqref{exam3-5}.
\end{proof}
\begin{rem}
Equivalent forms of the identities in Theorem \ref{thm-3} were given by Calinescu, Milas and Penn \cite[Eqs. (6.1)--(6.5)]{CMP}. They provided detailed proof for \eqref{exam3-2} and pointed out that other identities follow in a similar way. The proof here is essentially the same with that used in the proof of \cite[Theorem 6.1]{CMP}.
\end{rem}

\subsection{Example 4}\label{sec-exam4}
The modular triples for this example are given in Table \ref{tab-4}.
\begin{table}[H]
\centering
\begin{tabular}{c|cc}

  \hline   
    \padedvphantom{I}{3ex}{3ex}
 $A$ & \multicolumn{2}{c}{$\begin{pmatrix} 4 & 1 \\ 1 & 1 \end{pmatrix}$} \\
\hline
  \padedvphantom{I}{3ex}{3ex}
  $B$ & $\begin{pmatrix} 0 \\ 1/2  \end{pmatrix}$  & $\begin{pmatrix} 2 \\ 1/2  \end{pmatrix}$
   \\
  ~~$C$ & $1/120$ & $49/120$ \\
  \hline
\end{tabular}
\\[2mm]
\caption{Modular triples for Example 4}\label{tab-4}
\end{table}

Recall that Zagier discovered the formulas \eqref{Zagier-exam4-1} and \eqref{Zagier-exam4-2} but did not prove them.  We state the following equivalent formulas and give a proof.
\begin{theorem}\label{thm-4}
We have
\begin{align}
\sum_{i,j\geq 0}\frac{q^{2i^2+ij+\frac{1}{2}j^2+\frac{1}{2}j}}{(q;q)_i(q;q)_j}&=\frac{(q^4,q^6,q^{10};q^{10})_\infty}{(q;q)_\infty}, \label{exam4-1} \\
\sum_{i,j\geq 0}\frac{q^{2i^2+ij+\frac{1}{2}j^2+2i+\frac{1}{2}j}}{(q;q)_i(q;q)_j}&=\frac{(q^2,q^8,q^{10};q^{10})_\infty}{(q;q)_\infty}. \label{exam4-2}
\end{align}
As a consequence, the Nahm sums $f_{A,B,C}(q^{120})$ for $(A,B,C)$ in Example 4 are modular functions on $\Gamma_1(28800)$.
\end{theorem}
\begin{proof}
We define
\begin{align}\label{exam4-F-defn}
F(u,v;q):=\sum_{i,j\geq 0} \frac{q^{2i^2+ij+\frac{1}{2}j^2}u^iv^j}{(q;q)_i(q;q)_j}.
\end{align}
Summing over $j$ first using \eqref{Euler}, we deduce that
\begin{align}\label{exam4-F}
F(u,v;q)=\sum_{i=0}^\infty \frac{q^{2i^2}u^i}{(q;q)_i}\sum_{j=0}^\infty \frac{q^{\frac{1}{2}(j^2-j)}(q^{i+\frac{1}{2}}v)^j}{(q;q)_j} =\sum_{i=0}^\infty \frac{q^{2i^2}u^i}{(q;q)_i}(-q^{i+\frac{1}{2}}v;q)_\infty.
\end{align}
Setting $(u,v)=(1,q^{\frac{1}{2}})$ in \eqref{exam4-F}, we have
\begin{align}
F(1,q^{\frac{1}{2}};q)=(-q;q)_\infty \sum_{i=0}^\infty \frac{q^{2i^2}}{(q^2;q^2)_i}.
\end{align}
Using \eqref{RR-1}  with $q$ replaced by $q^2$, we obtain \eqref{exam4-1}.

Next, setting $(u,v)=(q^2,q^{\frac{1}{2}})$ in \eqref{exam4-F}, we deduce that
\begin{align}
F(q^2,q^{\frac{1}{2}};q)=(-q;q)_\infty \sum_{i=0}^\infty \frac{q^{2i^2+2i}}{(q^2;q^2)_i}.
\end{align}
Using \eqref{RR-2} with $q$ replaced by $q^2$, we obtain \eqref{exam4-2}.
\end{proof}
\begin{rem}\label{rem-exam4}
This example was implicitly proved by Vlasenko and Zwegers \cite{VZ}.
Let $r=1$, $(A,B,C)=(2,0,-1/60)$ and $(2,1,11/60)$ in \eqref{eq-VZ} and using \eqref{RR-1}--\eqref{RR-2}, we obtain Theorem \ref{thm-4}.
\end{rem}

\subsection{Example 5}\label{sec-exam5}
The modular triples for this example are given in Table \ref{tab-5}.
\begin{table}[H]
\centering
\begin{tabular}{c|cc}
  \hline
    \padedvphantom{I}{3ex}{3ex}
  $A$ & \multicolumn{2}{c}{$\begin{pmatrix} 1/3 & -1/3 \\ -1/3 & 4/3 \end{pmatrix}$} \\
  \hline
    \padedvphantom{I}{3ex}{3ex}
  $B$ & $\begin{pmatrix} -1/6 \\ 2/3 \end{pmatrix}$ & $\begin{pmatrix} 1/2 \\ 0 \end{pmatrix}$ \\
  ~~$C$ & $3/40$ & $1/120$ \\
  \hline
\end{tabular}
\\[2mm]
\caption{Modular triples for Example 5}\label{tab-5}
\end{table}


We find the desired identities with the help of Maple. But we are not able to prove them at this stage. Hence we state these identities as the following conjecture.
\begin{conj}\label{conj-exam5}
We have
\begin{align}
&\sum_{i,j\geq 0}\frac{q^{\frac{1}{2}i^2-ij+2j^2-\frac{1}{2}i+2j}}{(q^3;q^3)_i(q^3;q^3)_j}  \nonumber \\ &=3\frac{J_6J_{45}J_{18,90}J_{27,90}}{J_3^2J_{90}^2}-\frac{J_{10}J_{1,30}J_{4,30}J_{5,30}J_{6,30}J_{11,30}J_{14,30}}{J_3J_{30}^5J_{3,30}},  \label{exam5-1} \\
&\sum_{i,j\geq 0}\frac{q^{\frac{1}{2}i^2-ij+2j^2+\frac{3}{2}i}}{(q^3;q^3)_i (q^3;q^3)_j}\nonumber \\
&=\frac{J_{10}J_{2,30}J_{5,30}J_{7,30}J_{8,30}J_{12,30}J_{13,30}}{J_3J_{30}^5J_{9,30}}+3q^2\frac{J_6J_{45}J_{9,90}J_{36,90}}{J_3^2J_{90}^2}.  \label{exam5-2}
\end{align}
\end{conj}
As a consequence of this conjecture, the Nahm sums $f_{A,B,C}(q^{120})$ for $(A,B,C)$ Example 5 are modular functions on $\Gamma_1(28800)$.

Below we discuss how this conjecture is formulated. We try to find the 3-dissections:
\begin{align}
\sum_{i,j\geq 0}\frac{q^{\frac{1}{2}i^2-ij+2j^2-\frac{1}{2}i+2j}}{(q^3;q^3)_i(q^3;q^3)_j}&=F_0(q^3)+qF_1(q^3)+q^2F_2(q^3), \\
\sum_{i,j\geq 0}\frac{q^{\frac{1}{2}i^2-ij+2j^2+\frac{3}{2}i}}{(q^3;q^3)_i (q^3;q^3)_j}&=G_0(q^3)+qG_1(q^3)+q^2G_2(q^3)
\end{align}
with $F_i(q), G_i(q)\in \mathbb{Z}[\![q]\!]$ ($i=1,2,3$).

It is easy to prove that
\begin{align*}
\frac{1}{2}i^2-ij+2j^2-\frac{1}{2}i+2j\equiv 0, 1 \pmod{3}, \quad \frac{1}{2}i^2-ij+2j^2+\frac{3}{2}i \equiv 0,2 \pmod{3}.
\end{align*}
Hence
\begin{align}
F_2(q)=0, \quad G_1(q)=0.
\end{align}
With the help of Maple, it appears that
\begin{align}
F_0(q)&=2\frac{J_2J_{15}J_{6,30}J_{9,30}}{J_1^2J_{30}^2}, \label{exam5-F0}\\
G_2(q)&=2\frac{J_2J_{15}J_{3,30}J_{12,30}}{J_1^2J_{30}^2} \label{exam5-G2}
\end{align}
and there are no single product representations for $F_1(q)$ and $G_0(q)$.
This then motivates us to subtract the original series from suitable multiples of $F_0(q^3)$ and $q^2G_2(q^3)$ so that the results have simple product representations. As a consequence, we find the identities stated in the above conjecture.

To prove Conjecture \ref{conj-exam5}, one might need to prove \eqref{exam5-F0} and \eqref{exam5-G2} first and then find representations for $F_1(q)$ and $G_0(q)$.

\subsection{Example 6}\label{sec-exam6}
The modular triples for this example are given in Table \ref{tab-6}.
\begin{table}[H]
\centering
\begin{tabular}{c|ccc}
  \hline
    \padedvphantom{I}{3ex}{3ex}
  $A$ &  \multicolumn{3}{c}{$\begin{pmatrix} 4 & 2 \\ 2 & 2 \end{pmatrix}$} \\
  \hline
    \padedvphantom{I}{3ex}{3ex}
  $B$ & $\begin{pmatrix}  0 \\ 0  \end{pmatrix}$ & $\begin{pmatrix} 1 \\ 0   \end{pmatrix}$ & $\begin{pmatrix} 2 \\ 1   \end{pmatrix}$  \\
  ~~$C$ & $-1/42$ & $5/42$ & $17/42$ \\
  \hline
\end{tabular}
\\[2mm]
\caption{Modular triples for Example 6}\label{tab-6}
\end{table}


\begin{theorem}\label{thm-6}
We have
\begin{align}
\sum_{i,j\geq 0} \frac{q^{2i^2+j^2+2ij}}{(q;q)_i(q;q)_j}&=\frac{(q^3,q^4,q^7;q^7)_\infty}{(q;q)_\infty}, \label{exam6-1} \\
\sum_{i,j\geq 0} \frac{q^{2i^2+j^2+2ij+i}}{(q;q)_i(q;q)_j} &=\frac{(q^2,q^5,q^7;q^7)_\infty}{(q;q)_\infty}, \label{exam6-2} \\
\sum_{i,j\geq 0} \frac{q^{2i^2+j^2+2ij+2i+j}}{(q;q)_i(q;q)_j}&=\frac{(q,q^6,q^7;q^7)_\infty}{(q;q)_\infty}. \label{exam6-3}
\end{align}
As a consequence, the Nahm sums $f_{A,B,C}(q^{42})$ for $(A,B,C)$ in Example 6 are modular functions on $\Gamma_1(1764)$.
\end{theorem}
As noted in \cite[Example 3.4.3]{LeeThesis}, these identities are special cases of the Andrews--Gordon identity \cite{Andrews1974,Gordon1961}, which states that for integers $k,s$ such that $k\geq 2$ and $1\leq s \leq k$,
\begin{align}
\sum_{n_1,\cdots,n_{k-1}\geq 0} \frac{q^{N_1^2+\cdots+N_{k-1}^2+N_s+\cdots +N_{k-1}}}{(q;q)_{n_1}\cdots (q;q)_{n_{k-1}}}  =\frac{(q^s,q^{2k+1-s},q^{2k+1};q^{2k+1})_\infty}{(q;q)_\infty} \label{AG}
\end{align}
where if $j\leq k-1$, $N_j=n_j+\cdots+n_{k-1}$ and $N_k=0$. This is a generalization of the Rogers--Ramanujan identities.

If we set $k=3$ and $s=3,2,1$, we obtain \eqref{exam6-1}, \eqref{exam6-2} and \eqref{exam6-3}, respectively.

\subsection{Example 7.}\label{sec-exam7}
The modular triples for this example are given in Table \ref{tab-7}.
\begin{table}[H]
\centering
\begin{tabular}{c|ccc}
  \hline
    \padedvphantom{I}{3ex}{3ex}
  $A$ &  \multicolumn{3}{c}{$\begin{pmatrix} 1/2 & -1/2 \\ -1/2 & 1 \end{pmatrix}$}  \\
  \hline
    \padedvphantom{I}{3ex}{3ex}
  $B$ & $\begin{pmatrix} 0 \\ 0  \end{pmatrix}$ & $\begin{pmatrix}  1/2\\ -1/2 \end{pmatrix}$ & $\begin{pmatrix} 1/2 \\ 0 \end{pmatrix}$ \\
  ~~$C$ & $-5/84$ & $1/21$ & $1/84$ \\
  \hline
\end{tabular}
\\[2mm]
\caption{Modular triples for Example 7}\label{tab-7}
\end{table}


\begin{theorem}\label{thm-7}
We have
\begin{align}
\sum_{i,j\geq 0}\frac{q^{i^2+2j^2-2ij}}{(q^4;q^4)_i(q^4;q^4)_j}&=\frac{J_4^3J_{56}J_{6,28}}{J_2^2 J_8  J_{8,56}J_{24,56}}+2q\frac{J_8J_{56}J_{8,56}}{J_4^2J_{4,56}}, \label{exam7-1} \\
\sum_{i,j\geq 0}\frac{q^{i^2+2j^2-2ij+2i-2j}}{(q^4;q^4)_i(q^4;q^4)_j}&= 2\frac{J_8J_{56}J_{24,56}}{J_4^2J_{12,56}}+q\frac{J_4^3J_{28}J_{10,56}J_{18,56}}{J_2^2J_8J_{56}J_{16,56}J_{24,56}}, \label{exam7-2} \\
\sum_{i,j\geq 0}\frac{q^{i^2+2j^2-2ij+2i}}{(q^4;q^4)_i (q^4;q^4)_j}&=\frac{J_4^3J_{28}J_{2,56}J_{26,56}}{J_2^2J_8J_{56}J_{8,56}J_{16,56}}+2q^3\frac{J_8J_{56}J_{16,56}}{J_4^2J_{20,56}}. \label{exam7-3}
\end{align}
As a consequence, the Nahm sums $f_{A,B,C}(q^{84})$ for $(A,B,C)$  in Example 8 are modular functions on $\Gamma_1(7056)$.
\end{theorem}
\begin{proof}
(1) We need to find a 2-dissection for the left side:
\begin{align}\label{exam7-1-dissection}
\sum_{i,j\geq 0}\frac{q^{i^2+2j^2-2ij}}{(q^4;q^4)_i(q^4;q^4)_j}=F_0(q^2)+qF_1(q^2), \quad F_0(q),F_1(q)\in \mathbb{Z}[\![q]\!].
\end{align}
Clearly, the parity of the exponent $i^2+2j^2-2ij$ is the same with $i$. Hence,
\begin{align}
F_0(q^2)=\sum_{i,j\geq 0}\frac{q^{(2i)^2+2j^2-2\cdot (2i)j}}{(q^4;q^4)_{2i}(q^4;q^4)_j}.
\end{align}
Therefore,
\begin{align}\label{exam7-1-F0}
F_0(q)&=\sum_{i,j\geq 0}\frac{q^{2i^2+j^2-2ij}}{(q^2;q^2)_{2i}(q^2;q^2)_j} =\sum_{i=0}^\infty \frac{q^{2i^2}}{(q^2;q^2)_{2i}}\sum_{j=0}^\infty \frac{q^{j^2-j}\cdot q^{(1-2i)j}}{(q^2;q^2)_j}  \nonumber \\
&=\sum_{i=0}^\infty \frac{q^{2i^2}}{(q^2;q^2)_{2i}}{(-q^{1-2i};q^2)_\infty} =(-q;q^2)_\infty \sum_{i=0}^\infty \frac{q^{i^2}(-q;q^2)_i}{(q^2;q^2)_{2i}}.
\end{align}
Substituting \eqref{Slater117} into \eqref{exam7-1-F0}, we obtain
\begin{align}\label{exam7-1-F0-result}
F_0(q)=\frac{J_2^3J_{28}J_{3,14}}{J_1^2J_4J_{4,28}J_{12,28}}.
\end{align}

Next, we have
\begin{align}
qF_1(q^2)=\sum_{i,j\geq 0}\frac{q^{(2i+1)^2+2j^2-2(2i+1)j}}{(q^4;q^4)_{2i+1}(q^4;q^4)_j}.
\end{align}
Thus
\begin{align}
F_1(q)&=\sum_{i,j\geq 0}\frac{q^{2i^2+2i+j^2-j-2ij}}{(q^2;q^2)_{2i+1}(q^2;q^2)_j}.
\end{align}
We have
\begin{align}\label{exam7-1-F1}
F_1(q^{\frac{1}{2}})&=\sum_{i,j\geq 0}\frac{q^{i^2+i}\cdot q^{(j^2-j)/2-ij}}{(q;q)_{2i+1}(q;q)_j} =\sum_{i=0}^\infty \frac{q^{i^2+i}}{(q;q)_{2i+1}}\sum_{j=0}^\infty \frac{q^{(j^2-j)/2}\cdot q^{-ij}}{(q;q)_j} \nonumber \\
&=\sum_{i=0}^\infty \frac{q^{i^2+i}}{(q;q)_{2i+1}}(-q^{-i};q)_\infty =2(-q;q)_\infty \sum_{i=0}^\infty \frac{q^{(i^2+i)/2}(-q;q)_i}{(q;q)_{2i+1}}.
\end{align}
Substituting \eqref{Slater80} into \eqref{exam7-1-F1}, we deduce that
\begin{align}\label{exam7-1-F1-result}
F_1(q^{\frac{1}{2}})=2\frac{J_2J_{14}J_{2,14}}{J_1^2J_{1,14}}.
\end{align}
Substituting \eqref{exam7-1-F0-result} and \eqref{exam7-1-F1-result} into \eqref{exam7-1-dissection}, we obtain \eqref{exam7-1}.

(2) Now we prove the second identity. Again, we need to make a 2-dissection:
\begin{align}\label{exam7-2-dissection}
\sum_{i,j\geq 0}\frac{q^{i^2+2j^2-2ij+2i-2j}}{(q^4;q^4)_i(q^4;q^4)_j}=G_0(q^2)+qG_1(q^2), \quad G_0(q),G_1(q)\in \mathbb{Z}[\![q]\!].
\end{align}
Note that the parity of the exponent $i^2+2j^2-2ij+2i-2j$  is the same with $i$. We deduce that
\begin{align}
G_0(q^2)&=\sum_{i,j\geq 0}\frac{q^{(2i^2)+2j^2-2\cdot(2i)\cdot j+4i-2j}}{(q^4;q^4)_{2i}(q^4;q^4)_j} =\sum_{i,j\geq 0}\frac{q^{4i^2+2j^2-4ij+4i-2j}}{(q^4;q^4)_{2i}(q^4;q^4)_j}.
\end{align}
We have
\begin{align}\label{exam7-2-G0}
G_0(q^{\frac{1}{2}})&=\sum_{i,j\geq 0}\frac{q^{i^2-ij+i+(j^2-j)/2}}{(q;q)_{2i}(q;q)_j} =\sum_{i=0}^\infty \frac{q^{i^2+i}}{(q;q)_{2i}}\sum_{j=0}^\infty \frac{q^{(j^2-j)/2}\cdot q^{-ij}}{(q;q)_j} \nonumber \\
&=\sum_{i=0}^\infty \frac{q^{i^2+i}}{(q;q)_{2i}}(-q^{-i};q)_\infty =2(-q;q)_\infty \sum_{i=0}^\infty \frac{q^{(i^2+i)/2}(-q;q)_i}{(q;q)_{2i}}.
\end{align}
Substituting \eqref{Slater81} into \eqref{exam7-2-G0}, we obtain
\begin{align}\label{exam7-2-G0-result}
G_0(q^{\frac{1}{2}})=2\frac{J_2J_{14}J_{6,14}}{J_1^2J_{3,14}}.
\end{align}
Similarly,
\begin{align}
qG_1(q^2)&=\sum_{i,j\geq 0}\frac{q^{(2i+1)^2+2j^2-2(2i+1)j+2(2i+1)-2j}}{(q^4;q^4)_{2i+1}(q^4;q^4)_j} =\sum_{i,j\geq 0}\frac{q^{4i^2-4ij+2j^2+8i-4j+3}}{(q^4;q^4)_{2i+1}(q^4;q^4)_j}.
\end{align}
Hence
\begin{align}\label{exam7-2-G1}
G_1(q)&=\sum_{i,j\geq 0}\frac{q^{2i^2+4i+1+j^2-2j-2ij}}{(q^2;q^2)_{2i+1}(q^2;q^2)_j} =\sum_{i=0}^\infty \frac{q^{2i^2+4i+1}}{(q^2;q^2)_{2i+1}}\sum_{j=0}^\infty \frac{q^{j^2-j}\cdot q^{-(2i+1)j}}{(q^2;q^2)_j} \nonumber \\
&=\sum_{i=0}^\infty \frac{q^{2i^2+4i+1}}{(q^2;q^2)_{2i+1}}(-q^{-2i-1};q^2)_\infty =(-q;q^2)_\infty \sum_{i=0}^\infty \frac{q^{i^2+2i}(-q;q^2)_{i+1}}{(q^2;q^2)_{2i+1}}.
\end{align}
Substituting \eqref{Slater119} into \eqref{exam7-2-G1}, we obtain
\begin{align}\label{exam7-2-G1-result}
G_1(q)=\frac{J_2^3J_{14}J_{5,28}J_{9,28}}{J_1^2J_4J_{28}J_{8,28}J_{12,28}}.
\end{align}
Now substituting \eqref{exam7-2-G0-result} and \eqref{exam7-2-G1-result} into \eqref{exam7-2-dissection}, we obtain \eqref{exam7-2}.

(3) It remains to prove \eqref{exam7-3}.  We make a 2-dissection:
\begin{align}\label{exam7-3-dissection}
\sum_{i,j\geq 0}\frac{q^{i^2+2j^2-2ij+2i}}{(q^4;q^4)_i (q^4;q^4)_j}=H_0(q^2)+qH_1(q^2).
\end{align}
Note that the parity of $i^2+2j^2-2ij+2i$ is the same with $i$. We have
\begin{align}
H_0(q^2)&=\sum_{i,j\geq 0}\frac{q^{4i^2+2j^2-4ij+4i}}{(q^4;q^4)_{2i}(q^4;q^4)_j},  \label{exam7-3-H0-defn}\\
qH_1(q^2)&=\sum_{i,j\geq 0}\frac{q^{4i^2+8i+3+2j^2-4ij-2j}}{(q^4,q^4)_{2i+1}(q^4;q^4)_j}. \label{exam7-3-H1-defn}
\end{align}
From \eqref{exam7-3-H0-defn} we have
\begin{align}\label{exam7-3-H0}
H_0(q)&=\sum_{i,j\geq 0}\frac{q^{2i^2+j^2-2ij+2i}}{(q^2;q^2)_{2i}(q^2;q^2)_j} =\sum_{i=0}^\infty \frac{q^{2i^2+2i}}{(q^2;q^2)_{2i}}\sum_{j=0}^\infty \frac{q^{j^2-j}\cdot q^{(1-2i)j}}{(q^2;q^2)_j} \nonumber \\
&=\sum_{i=0}^\infty \frac{q^{2i^2+2i}}{(q^2;q^2)_{2i}} (-q^{1-2i};q^2)_\infty =(-q;q^2)_\infty \sum_{i=0}^\infty \frac{q^{i^2+2i}(-q;q^2)_i}{(q^2;q^2)_{2i}}.
\end{align}
Substituting \eqref{Slater118} into \eqref{exam7-3-H0}, we obtain
\begin{align}\label{exam7-3-H0-result}
H_0(q)=\frac{J_2^3J_{14}J_{1,28}J_{13,28}}{J_1^2J_4J_{28}J_{4,28}J_{8,28}}.
\end{align}
From \eqref{exam7-3-H1-defn} we have
\begin{align}\label{exam7-3-H1}
H_1(q^{\frac{1}{2}})&=q^{\frac{1}{2}}\sum_{i,j\geq 0}\frac{q^{i^2+2i-ij+(j^2-j)/2}}{(q;q)_{2i+1}(q;q)_j} =q^{\frac{1}{2}}\sum_{i=0}^\infty \frac{q^{i^2+2i}}{(q;q)_{2i+1}} \sum_{j=0}^\infty \frac{q^{(j^2-j)/2}\cdot q^{-ij}}{(q;q)_j} \nonumber \\
&=q^{\frac{1}{2}}\sum_{i=0}^\infty \frac{q^{i^2+2i}}{(q;q)_{2i+1}}(-q^{-i};q)_\infty =2q^{\frac{1}{2}}(-q;q)_\infty \sum_{i=0}^\infty \frac{q^{(i^2+3i)/2}(-q;q)_i}{(q;q)_{2i+1}}.
\end{align}
Substituting \eqref{Slater82} into \eqref{exam7-3-H1}, we obtain
\begin{align}\label{exam7-3-H1-result}
H_1(q^{\frac{1}{2}})&=2q^{\frac{1}{2}}\frac{J_2J_{14}J_{4,14}}{J_1^2J_{5,14}}.
\end{align}
Substituting \eqref{exam7-3-H0-result} and \eqref{exam7-3-H1-result} into \eqref{exam7-3-dissection}, we obtain \eqref{exam7-3}.

From the identity \eqref{exam7-1} we have
\begin{align}\label{f-g1g2}
f_{A,(0,0)^\mathrm{T},-5/84}(q^4)=g_1(\tau)+2g_2(\tau)
\end{align}
where
\begin{align}
g_1(\tau)&=\frac{\eta_{56,4}(\tau)\eta_{56,12}(\tau)\eta_{56,20}(\tau)\eta_{56,28}(\tau)}
{\eta_{56,2}^2(\tau)\eta_{56,6}(\tau)\eta_{56,8}(\tau)\eta_{56,10}^2(\tau)\eta_{56,14}^2(\tau)\eta_{56,18}^2(\tau)\eta_{56,22}(\tau)\eta_{56,24}(\tau)\eta_{56,26}^2(\tau)}, \\
g_2(\tau)&=\frac{1}{\eta_{56,4}^3(\tau)\eta_{56,12}^2(\tau)\eta_{56,16}(\tau)\eta_{56,20}^2(\tau)\eta_{56,24}(\tau)\eta_{56,28}(\tau)}.
\end{align}
Now using the method in \cite{Garvan-Liang} (see Section \ref{sec-modularity}), we check by Maple that both $g_1(21\tau)$ and $g_2(21\tau)$ are modular functions on $\Gamma_1(7056)$. For instance, the Maple code for checking the modularity of $g_2(21\tau)$ is as follows:\\
\begin{tcolorbox}
{$ \displaystyle \texttt{>\,}  \mathit{L}:= [[56\cdot 21, 4\cdot 21, -3], [56\cdot 21, 12\cdot 21, -2], [56\cdot 21, 16\cdot 21, -1], [56\cdot  21, 20\cdot 21, -2], [56\cdot 21, 24\cdot 21, -1], [56\cdot 21, 28\cdot 21, -1]] $}\\
{$ \displaystyle \texttt{>\,}  \mathit{xprint:= true:}$}\\
{$\displaystyle \texttt{>\,} \mathit{Gamma1ModFunc}(\mathit{L}, 7056)$}
\begin{verbatim}
* starting Gamma1ModFunc with L=[[1176, 84, -2], [1176, 168, -1],
[1176, 252, -2], [1176, 420, -3], [1176, 504, -1],
[1176, 588, -1]] and N=7056
All n are divisors of 7056
val0=-10
which is even.
valinf=128
which is even.
It IS a modfunc on Gamma1(7056)
\end{verbatim}
$$1$$
\end{tcolorbox}
From \eqref{f-g1g2} we know that $f_{A,(0,0)^\mathrm{T},-5/84}(q^{84})$ is a modular function on $\Gamma_1(7056)$.

Similarly, we have checked that, after multiplication by $q^{4/21}$ (resp.\ $q^{1/21}$) and then replacing $q$ by $q^{21}$,  the two infinite products in \eqref{exam7-2} (resp.\ \eqref{exam7-3}) are  modular functions on $\Gamma_1(7056)$. This proves the desired assertion.
\end{proof}


\subsection{Example 8}\label{sec-exam8}
The modular triples for this example are given in Table \ref{tab-8} \footnote{Here the last value of $C$ has been corrected. In \cite[Table 2]{Zagier} it was written as $25/168$.}.
\begin{table}[H]
\centering
\begin{tabular}{c|ccc}
  \hline
    \padedvphantom{I}{3ex}{3ex}
  $A$ &  \multicolumn{3}{c}{$\begin{pmatrix} 3/2 & 1 \\ 1 & 2 \end{pmatrix}$}   \\
  \hline
    \padedvphantom{I}{3ex}{3ex}
  $B$ & $\begin{pmatrix} -1/2 \\ 0  \end{pmatrix}$ & $\begin{pmatrix} 0 \\0  \end{pmatrix}$ & $\begin{pmatrix} 1/2 \\ 1  \end{pmatrix}$ \\
  ~~$C$ & $1/168$ & $-5/168$ & $5/21$\\
  \hline
\end{tabular}
\\[2mm]
\caption{Modular triples for Example 8}\label{tab-8}
\end{table}


\begin{theorem}\label{thm-8}
We have
\begin{align}
\sum_{i,j\geq 0}\frac{q^{3i^2+4j^2+4ij-2i}}{(q^4;q^4)_i (q^4;q^4)_j}&=\sum_{i,j\geq 0}\frac{q^{i^2+4ij+4j^2+4j}}{(q^4;q^4)_i(q^8;q^8)_j}=\frac{J_{14}J_{28}^2J_{2,28}}{J_{1,28}J_{4,28}J_{8,28}J_{13,28}}, \label{exam8-1} \\
\sum_{i,j\geq 0}\frac{q^{3i^2+4j^2+4ij}}{(q^4;q^4)_i (q^4;q^4)_j}&=\sum_{i,j\geq 0}\frac{q^{i^2+4ij+4j^2+2i}}{(q^4;q^4)_i(q^8;q^8)_j}=\frac{J_{14}J_{28}^2J_{6,28}}{J_{3,28}J_{4,28}J_{11,28}J_{12,28}},  \label{exam8-2} \\
\sum_{i,j\geq 0}\frac{q^{3i^2+4j^2+4ij+2i+4j}}{(q^4;q^4)_i (q^4;q^4)_j}&=\sum_{i,j\geq 0}\frac{q^{i^2+4ij+4j^2+4i+4j}}{(q^4;q^4)_i(q^8;q^8)_j}=\frac{J_{14}J_{28}^2J_{10,28}}{J_{5,28}J_{8,28}J_{9,28}J_{12,28}}.  \label{exam8-3}
\end{align}
As a consequence, for $A$ in Example 8, the Nahm sums $f_{A,(-1/2,0)^\mathrm{T},1/168}(q^{168})$ and $f_{A,(0,0)^\mathrm{T},-5/168}(q^{168})$ are modular functions on $\Gamma_1(28224)$, and $f_{A,(1/2,1)^\mathrm{T},5/21}(q^{84})$  is a modular function on $\Gamma_1(14112)$.
\end{theorem}
Here the second double sum expression in each identity was found by the author in the proof. This example is more technical than the others since we cannot eliminate the variables $i$ or $j$ from the left side in an easy way. To achieve this, we need to use an integral method mentioned in Section \ref{sec-pre} to give the new double sum representations stated above.

\begin{proof}[Proof of Theorem \ref{thm-8}]
We define
\begin{align}
F(u,v;q):=\sum_{i,j\geq 0}\frac{q^{2i^2+(i+2j)^2}u^iv^j}{(q^4;q^4)_i(q^4;q^4)_j}.
\end{align}
By \eqref{Euler} and \eqref{Jacobi}, we have
\begin{align}\label{exam8-1-F-defn}
F(u,v;q)&=\oint \sum_{i=0}^\infty \frac{q^{2i^2}z^iu^i}{(q^4;q^4)_i} \sum_{j=0}^\infty \frac{z^{2j}v^j}{(q^4;q^4)_j} \sum_{k=-\infty}^\infty q^{k^2}z^{-k} \frac{\mathrm{d}z}{2\pi iz} \nonumber \\
&=\oint \frac{(-uzq^2;q^4)_\infty (-qz,-q/z,q^2;q^2)_\infty}{(vz^2;q^4)_\infty} \frac{\mathrm{d}z}{2\pi iz}.
\end{align}
Now we discuss the identities one by one.

(1) Setting $(u,v)=(q^{-2},1)$ in \eqref{exam8-1-F-defn}, we have
\begin{align}
F(q^{-2},1;q)&=\oint \frac{(-z;q^4)_\infty (-qz,-q/z,q^2;q^2)_\infty}{(z^2;q^4)_\infty} \frac{\mathrm{d}z}{2\pi iz} \nonumber \\
&=\oint \frac{(-z;q^4)_\infty(-qz,-q/z,q^2;q^2)_\infty}{(z^2;q^8)_\infty (z^2q^4;q^8)_\infty} \frac{\mathrm{d}z}{2\pi iz} \nonumber \\
&=\oint \frac{(-qz,-q/z,q^2;q^2)_\infty}{(z;q^4)_\infty (z^2q^4;q^8)_\infty} \frac{\mathrm{d}z}{2\pi iz}. \label{exam8-new-1}
\end{align}
Applying \eqref{Euler} and \eqref{Jacobi} to \eqref{exam8-new-1}, we deduce that
\begin{align}
F(q^{-2},1;q)&=\oint \sum_{i=0}^\infty \frac{z^i}{(q^4;q^4)_i} \sum_{j=0}^\infty \frac{z^{2j}q^{4j}}{(q^8;q^8)_j} \sum_{k=-\infty}^\infty q^{k^2}z^{-k} \frac{\mathrm{d}z}{2\pi iz} \nonumber \\
&=\sum_{i,j\geq 0}\frac{q^{(i+2j)^2+4j}}{(q^4;q^4)_i(q^8;q^8)_j}.
\end{align}
Now we have
\begin{align}
&\sum_{i,j\geq 0}\frac{q^{(i+2j)^2+4j}}{(q^4;q^4)_i(q^8;q^8)_j}=\sum_{i\geq 0} \frac{q^{i^2}}{(q^4;q^4)_i} \sum_{j\geq 0} \frac{q^{4j^2-4j} \cdot q^{(4i+8)j}}{(q^8;q^8)_j} \nonumber \\
&=\sum_{i=0}^\infty \frac{q^{i^2}(-q^{4i+8};q^8)_\infty}{(q^4;q^4)_i}=R_0(q)+R_1(q). \label{exam8-R0R1}
\end{align}
Here $R_0(q)$ and $R_1(q)$ correspond to sums over even and odd values of $i$, respectively. To be precise,
\begin{align}
R_0(q)&=\sum_{i=0}^\infty \frac{q^{4i^2}(-q^{8i+8};q^8)_\infty}{(q^8;q^8)_{2i}}, \label{exam8-R0}\\
R_1(q)&=\sum_{i=0}^\infty \frac{q^{4i^2+4i+1}(-q^{8i+12};q^8)_\infty}{(q^4;q^4)_{2i+1}}. \label{exam8-R1}
\end{align}
We have
\begin{align}
R_0(q^{\frac{1}{4}})&=\sum_{i=0}^\infty \frac{q^{i^2}(-q^{2i+2};q^2)_\infty}{(q;q)_{2i}} =(-q^2;q^2)_\infty \sum_{i=0}^\infty \frac{q^{i^2}}{(q;q)_{2i}(-q^2;q^2)_i} \nonumber \\
&=(-q^2;q^2)_\infty \sum_{i=0}^\infty \frac{q^{i^2}(-q;q^2)_i}{(q^2;q^2)_{2i}}=\frac{J_{14}J_{3,28}J_{8,28}J_{11,28}}{J_1J_{28}^3}. \label{exam8-1-R0-result}
\end{align}
Here for the last equality we used \eqref{Slater117}.

Similarly,
\begin{align}
R_1(q^{\frac{1}{4}})&= q^{\frac{1}{4}}\sum_{i=0}^\infty  \frac{q^{i^2+i}(-q^{2i+3};q^2)_\infty} {(q;q)_{2i+1}}=q^{\frac{1}{4}}(-q;q^2)_\infty
\sum_{i=0}^\infty \frac{q^{i^2+i}}{(q;q)_{2i+1}(-q;q^2)_{i+1}}\nonumber\\
&=q^{\frac{1}{4}}(-q;q^2)_\infty \sum_{i=0}^\infty \frac{q^{i^2+i}(-q^2;q^2)_i}{(q^2;q^2)_{2i+1}} =q^{\frac{1}{4}}\frac{J_{14}J_{4,28}J_{6,28}J_{10,28}}{J_1J_{28}^3}. \label{exam8-1-R1-result}
\end{align}
Here for the last equality we used \eqref{Slater80} with $q$ replaced by $q^2$.

Substituting \eqref{exam8-1-R0-result} and \eqref{exam8-1-R1-result} into \eqref{exam8-R0R1}, we get the second equality in \eqref{exam8-1} after simple verification using the method in \cite{Garvan-Liang}.

(2) Setting $(u,v)=(1,1)$ in \eqref{exam8-1-F-defn}, we deduce that
\begin{align}
F(1,1;q)&=\oint \frac{(-q^2z;q^4)_\infty (-qz,-q/z,q^2;q^2)_\infty}{(z^2;q^4)_\infty} \frac{\mathrm{d}z}{2\pi iz} \nonumber \\
&=\oint \frac{(-q^2z;q^4)_\infty (-qz,-q/z,q^2;q^2)_\infty}{(q^4z^2,z^2;q^8)_\infty} \frac{\mathrm{d}z}{2\pi iz} \nonumber \\
&=\oint \frac{(-qz,-q/z,q^2;q^2)_\infty}{(q^2z;q^4)_\infty (z^2;q^8)_\infty} \frac{\mathrm{d}z}{2\pi iz}. \label{exam8-2-new}
\end{align}
Using \eqref{Euler} and \eqref{Jacobi} and arguing similarly as in (1), we get the first equality in \eqref{exam8-2}.

Next, we have
\begin{align}
&\sum_{i,j\geq 0}\frac{q^{i^2+4ij+4j^2+2i}}{(q^4;q^4)_i(q^8;q^8)_j}=\sum_{i\geq 0}\frac{q^{i^2+2i}}{(q^4;q^4)_i} \sum_{j\geq 0} \frac{q^{4j^2-4j}\cdot q^{(4i+4)j}}{(q^8;q^8)_j} \nonumber \\
&=\sum_{i\geq 0} \frac{q^{i^2+2i}(-q^{4i+4};q^8)_\infty}{(q^4;q^4)_i}=S_0(q)+S_1(q). \label{exam8-2-S0S1}
\end{align}
Here $S_0(q)$ and $S_1(q)$ correspond to sums over even and odd values of $i$, respectively. To be precise,
\begin{align}
S_0(q)&=\sum_{i=0}^\infty \frac{q^{4i^2+4i}(-q^{8i+4};q^8)_\infty}{(q^4;q^4)_{2i}}, \\
S_1(q)&=\sum_{i=0}^\infty \frac{q^{4i^2+8i+3}(-q^{8i+8};q^8)_\infty}{(q^4;q^4)_{2i+1}}.
\end{align}
We have
\begin{align}
S_0(q^{\frac{1}{4}})&=\sum_{i=0}^\infty \frac{q^{i^2+i}(-q^{2i+1};q^2)_\infty}{(q;q)_{2i}}=(-q;q^2)_\infty \sum_{i=0}^\infty \frac{q^{i^2+i}}{(q;q)_{2i} (-q;q^2)_i} \nonumber \\
&=(-q;q^2)_\infty\sum_{i=0}^\infty \frac{q^{i^2+i}(-q^2;q^2)_i}{(q^2;q^2)_{2i}} =\frac{J_{14}J_{2,28}J_{10,28}J_{12,28}}{J_1J_{28}^3}. \label{exam8-2-S0-result}
\end{align}
Here for the last equality we used \eqref{Slater81} with $q$ replaced by $q^2$.

Similarly,
\begin{align}
S_1(q^{\frac{1}{4}})&=q^{\frac{3}{4}}\sum_{i=0}^\infty \frac{q^{i^2+2i}(-q^{2i+2};q^2)_\infty}{(q;q)_{2i+1}}=q^{\frac{3}{4}}(-q^2;q^2)_\infty \sum_{i=0}^\infty \frac{q^{i^2+2i}}{(q;q)_{2i+1}(-q^2;q^2)_i} \nonumber \\
&=q^{\frac{3}{4}}(-q^2;q^2)_\infty \sum_{i=0}^\infty \frac{q^{i^2+2i}(-q;q^2)_{i+1}}{(q^2;q^2)_{2i+1}}=q^{\frac{3}{4}}\frac{J_{14}J_{4,28}J_{5,28}J_{9,28}}{J_1J_{28}^3}. \label{exam8-2-S1-result}
\end{align}
Here for the last equality we used \eqref{Slater119}.

Substituting \eqref{exam8-2-S0-result} and \eqref{exam8-2-S1-result} into \eqref{exam8-2-S0S1}, we obtain the second equality in \eqref{exam8-2} after simple verifications using the method in \cite{Garvan-Liang}.

(3) Setting $(u,v)=(q^2,q^4)$ in \eqref{exam8-1-F-defn}, we have
\begin{align}
F(q^2,q^4;q)&=\oint \frac{(-q^4z;q^4)_\infty (-qz,-q/z,q^2;q^2)_\infty}{(q^4z^2;q^4)_\infty}\frac{\mathrm{d}z}{2\pi iz} \nonumber \\
&=\oint \frac{(-q^4z;q^4)_\infty(-qz,-q/z,q^2;q^2)_\infty}{(q^8z^2;q^8)_\infty (q^4z^2;q^8)_\infty} \frac{\mathrm{d}z}{2\pi iz} \nonumber \\
&=\oint \frac{(-qz,-q/z,q^2;q^2)_\infty}{(q^4z;q^4)_\infty (q^4z^2;q^8)_\infty} \frac{\mathrm{d}z}{2\pi iz}. \label{exam8-3-new}
\end{align}
Applying \eqref{Euler} and \eqref{Jacobi} to \eqref{exam8-3-new} and arguing similarly as in (1), we obtain the first equality in \eqref{exam8-3}.

Next, we have
\begin{align}
&\sum_{i,j\geq 0}\frac{q^{i^2+4ij+4j^2+4i+4j}}{(q^4;q^4)_i(q^8;q^8)_j}=\sum_{i\geq 0}\frac{q^{i^2+4i}}{(q^4;q^4)_i}\sum_{j\geq 0}\frac{q^{4j^2-4j}\cdot q^{(4i+8)j}}{(q^8;q^8)_j} \nonumber \\
&=\sum_{i\geq 0}\frac{q^{i^2+4i}(-q^{4i+8};q^8)_\infty}{(q^4;q^4)_i}=T_0(q)+T_1(q). \label{exam8-3-T0T1}
\end{align}
Here $T_0(q)$ and $T_1(q)$ correspond to sums over even and odd values of $i$, respectively. To be precise,
\begin{align}
T_0(q)&=\sum_{i=0}^\infty \frac{q^{4i^2+8i}(-q^{8i+8};q^8)_\infty}{(q^4;q^4)_{2i}},  \\
T_1(q)&=\sum_{i=0}^\infty \frac{q^{4i^2+12i+5}(-q^{8i+12};q^8)_\infty}{(q^4;q^4)_{2i+1}}.
\end{align}
We have
\begin{align}
T_0(q^{\frac{1}{4}})&=\sum_{i=0}^\infty \frac{q^{i^2+2i}(-q^{2i+2};q^2)_\infty}{(q;q)_{2i}} =(-q^2;q^2)_\infty \sum_{i=0}^\infty \frac{q^{i^2+2i}}{(q;q)_{2i}(-q^2;q^2)_i} \nonumber \\
&=(-q^2;q^2)_\infty \sum_{i=0}^\infty \frac{q^{i^2+2i}(-q;q^2)_i}{(q^2;q^2)_{2i}} =\frac{J_{14}J_{1,28}J_{12,28}J_{13,28}}{J_1J_{28}^3}. \label{exam8-3-T0-result}
\end{align}
Here for the last equality we used \eqref{Slater118}.

Similarly,
\begin{align}
T_1(q^{\frac{1}{4}})&=q^{\frac{5}{4}}\sum_{i=0}^\infty \frac{q^{i^2+3i}(-q^{2i+3};q^2)_\infty}{(q;q)_{2i+1}}=q^{\frac{5}{4}}(-q;q^2)_\infty \sum_{i=0}^\infty \frac{q^{i^2+3i}}{(q;q)_{2i+1}(-q;q^2)_{i+1}} \nonumber \\
&=q^{\frac{5}{4}}(-q;q^2)_\infty \sum_{i=0}^\infty \frac{q^{i^2+3i}(-q^2;q^2)_i}{(q^2;q^2)_{2i+1}}=q^{\frac{5}{4}}\frac{J_{14}J_{2,28}J_{6,28}J_{8,28}}{J_1J_{28}^3}. \label{exam8-3-T1-result}
\end{align}
Substituting \eqref{exam8-3-T0-result} and \eqref{exam8-3-T1-result} into \eqref{exam8-3-T0T1}, we obtain the second equality in \eqref{exam8-3} after simple verifications using the method in \cite{Garvan-Liang}.
\end{proof}

We now present a different proof for the product representations in Theorem \ref{thm-8} without using the second double sum expression in each identity. This is actually the first proof we found when doing this project. The computations in this proof are a bit more complicated, but it tells us that the theorem can also be reduced to other identities in Slater's list. Here we follow the techniques in the author's work \cite{Wang}.

\begin{proof}[Second proof of the product representations in Theorem \ref{thm-8}]
(1) We rewrite \eqref{exam8-new-1} as
\begin{align}
F(q^{-2},1;q)=(q^2;q^4)_\infty \oint \frac{(-qz,-q^3z,-q/z,-q^3/z,q^4;q^4)_\infty}{(z,q^2z,-q^2z;q^4)_\infty} \frac{\mathrm{d}z}{2\pi iz}. \label{exam8-proof-1}
\end{align}
To prove the second equality in \eqref{exam8-1}, we use \eqref{exam8-proof-1}. Applying Lemma \ref{lem-integral} with $q$ replaced by $q^4$ and
$$(A,B,C,D)=(2,2,3,0), \quad (a_1,a_2)=(b_1,b_2)=(-q,-q^3), \quad (c_1,c_2,c_3)=(1,q^2,-q^2),$$
we deduce that
\begin{align}\label{exam8-1-F}
F(q^{-2},1;q)=(q^2;q^4)_\infty \left(R_1(q)+R_2(q)+R_3(q)\right),
\end{align}
where
\begin{align}
R_1(q)&:=\frac{(-q,-q,-q^3,-q^3;q^4)_\infty}{(q^2,-q^2;q^4)_\infty} \sum_{n=0}^\infty \frac{q^{2n^2+2n}(-q^3,-q;q^4)_n}{(q^4,-q,-q^3,q^2,-q^2;q^4)_n} \nonumber \\
&=\frac{(-q;q^2)_\infty^2}{(q^4;q^8)_\infty} \sum_{n=0}^\infty \frac{q^{2n^2+2n}}{(q^4;q^4)_n (q^4;q^8)_n}, \label{new-exam8-1-R1} \\
R_2(q)&:=\frac{(-q^3,-q^5,-q^{-1},-q;q^4)_\infty}{(q^{-2},-1;q^4)_\infty} \sum_{n=0}^\infty \frac{q^{2n^2+4n}(-q^5,-q^3;q^4)_n}{(q^4,-q^3,-q^5,q^6,-q^4;q^4)_n} \nonumber \\
&=-\frac{1}{2}q\frac{(-q;q^2)_\infty^2}{(q^2;q^4)_\infty (-q^4;q^4)_\infty} \sum_{n=0}^\infty \frac{q^{2n^2+4n}}{(q^2;q^4)_{n+1}(q^8;q^8)_n}, \label{new-exam8-1-R2} \\
R_3(q)&:=\frac{(q^3,q^5,q^{-1},q;q^4)_\infty}{(-1,-q^{-2};q^4)_\infty} \sum_{n=0}^\infty \frac{(-1)^nq^{2n^2+4n}(q^5,q^3;q^4)_n}{(q^4,q^3,q^5,-q^4,-q^6;q^4)_n} \nonumber \\
&=-\frac{1}{2}q \frac{(q;q^2)_\infty^2}{(-q^2;q^2)_\infty} \sum_{n=0}^\infty \frac{(-1)^nq^{2n^2+4n}}{(-q^2;q^4)_{n+1}(q^8;q^8)_n}. \label{new-exam8-1-R3}
\end{align}
Substituting the identity \eqref{Slater81} with $q$ replaced by $q^4$ into \eqref{new-exam8-1-R1}, we deduce that
\begin{align}
R_1(q)=\frac{J_2^4J_8J_{56}J_{24,56}}{J_1^2J_4^4J_{12,56}}.  \label{new-exam8-1-R1-result}
\end{align}
Substituting the identity  \eqref{Slater119}  with $q$ replaced by $q^2$ (resp.\ $-q^2$)  into \eqref{new-exam8-1-R2} (resp.\ \eqref{new-exam8-1-R3}), we deduce that
\begin{align}
R_2(q)&=-\frac{1}{2}q\frac{J_2^2J_4J_{8,56}J_{10,28}}{J_1^2J_8^2J_{56}}, \label{new-exam8-1-R2-result}  \\ R_3(q)&=-\frac{1}{2}q\frac{J_1^2J_{14}J_{2,28}J_{6,28}}{J_2J_4J_{28}J_{4,28}J_{12,28}}. \label{new-exam8-1-R3-result}
\end{align}
Now substituting \eqref{new-exam8-1-R1-result}--\eqref{new-exam8-1-R3-result} into \eqref{exam8-1-F}, using the method in \cite{Garvan-Liang}, it is easy to verify that the second equality of \eqref{exam8-1} holds.

(2) From \eqref{exam8-2-new} we have
\begin{align}
F(1,1;q)=(q^2;q^4)_\infty \oint \frac{(-qz,-q^3z,-q/z,-q^3/z,q^4;q^4)_\infty}{(z,-z,q^2z;q^4)_\infty} \frac{\mathrm{d}z}{2\pi iz}.  \label{exam8-2-F-start}
\end{align}
Applying Lemma \ref{lem-integral} with $q$ replaced by $q^4$ and
\begin{align*}
(A,B,C,D)=(2,2,3,0), \quad (a_1,a_2)=(b_1,b_2)=(-q,-q^3), \quad (c_1,c_2,c_3)=(1,-1,q^2),
\end{align*}
we deduce that
\begin{align}\label{exam8-2-F}
F(1,1;q)=(q^2;q^4)_\infty \left(S_1(q)+S_2(q)+S_3(q)\right),
\end{align}
where
\begin{align}
S_1(q)&=\frac{(-q,-q,-q^3,-q^3;q^4)_\infty}{(-1,q^2;q^4)_\infty} \sum_{n=0}^\infty \frac{q^{2n^2+4n}(-q^3,-q;q^4)_n}{(q^4,-q^4,q^2,-q,-q^3;q^4)_n}  \nonumber \\
&=\frac{1}{2}\frac{(-q;q^2)_\infty^2}{(q^2,-q^4;q^4)_\infty} \sum_{n=0}^\infty \frac{q^{2n^2+4n}}{(q^2;q^4)_n(q^8;q^8)_n},  \label{new-exam8-2-S1}\\
S_2(q)&=\frac{(q,q^3;q^4)_\infty^2}{(-1,-q^2;q^4)_\infty} \sum_{n=0}^\infty \frac{(-1)^nq^{2n^2+4n}(q,q^3;q^4)_n}{(q^4,q,q^3,-q^4,-q^2;q^4)_n} \nonumber \\
&=\frac{1}{2}\frac{(q;q^2)_\infty^2}{(-q^2;q^2)_\infty} \sum_{n=0}^\infty \frac{(-1)^nq^{2n^2+4n}}{(-q^2;q^4)_n(q^8;q^8)_n}, \label{new-exam8-2-S2} \\
S_3(q)&=\frac{(-q^3,-q^5,-q^{-1},-q;q^4)_\infty}{(-q^{-2},q^{-2};q^4)_\infty} \sum_{n=0}^\infty \frac{(-q^5,-q^3;q^4)_nq^{2n^2+6n}}{(q^4,-q^3,-q^5,-q^6,q^6;q^4)_n} \nonumber \\
&=-q^3\frac{(-q;q^2)_\infty^2}{(q^4;q^8)_\infty} \sum_{n=0}^\infty \frac{q^{2n^2+6n}}{(q^4;q^4)_n(q^4;q^8)_{n+1}}. \label{new-exam8-2-S3}
\end{align}
Substituting \eqref{Slater118} with $q$ replaced by $q^2$ (resp.\ $-q^2$) into \eqref{new-exam8-2-S1} (resp.\ \eqref{new-exam8-2-S2}), we deduce that
\begin{align}
S_1(q)&=\frac{1}{2}\frac{J_2^2J_4J_{28}J_{2,56}J_{26,56}}{J_1^2J_8J_{56}J_{8,56}J_{16,56}}, \label{new-exam8-2-S1-result} \\
S_2(q)&=\frac{1}{2}\frac{J_1^2J_{14}J_{6,28}J_{10,28}}{J_2J_4J_{28}J_{8,28}J_{12,28}}. \label{new-exam8-2-S2-result}
\end{align}
Substituting \eqref{Slater82} (with $q$ replaced by $q^4$) into \eqref{new-exam8-2-S3}, we obtain
\begin{align}
S_3(q)=-q^3\frac{J_2^4J_8^2J_{56}^3}{J_1^2J_4^4J_{8,56}J_{20,56}J_{24,56}}. \label{new-exam8-2-S3-result}
\end{align}
Now substituting \eqref{new-exam8-2-S1-result}--\eqref{new-exam8-2-S3-result} into \eqref{exam8-2-F}, using the method in \cite{Garvan-Liang}, it is easy to verify that the second equality in \eqref{exam8-2} holds.

(3) From \eqref{exam8-3-new} we have
\begin{align}
F(q^2,q^4;q)=(q^2;q^4)_\infty \oint \frac{(-qz,-q^3z,-q/z,-q^3/z;q^4)_\infty}{(q^4z,q^2z,-q^2z;q^4)_\infty}\frac{\mathrm{d}z}{2\pi iz}. \label{exam8-3-F-start}
\end{align}
Applying Lemma \ref{lem-integral} with
\begin{align*}
(A,B,C,D)=(2,2,3,0),  (a_1,a_2)=(b_1,b_2)=(-q,-q^3), (c_1,c_2,c_3)=(q^4,q^2,-q^2),
\end{align*}
we deduce that
\begin{align}\label{exam8-3-F}
F(q^2,q^4;q)=(q^2;q^4)_\infty (T_1(q)+T_2(q)+T_3(q)),
\end{align}
where
\begin{align}
T_1(q)&=\frac{(-q^5,-q^7,-q^{-3},-q^{-1};q^4)_\infty}{(q^{-2},-q^{-2};q^4)_\infty} \sum_{n=0}^\infty \frac{q^{2n^2+2n}(-q^7,-q^5;q^4)_n}{(q^4,-q^5,-q^7,q^6,-q^6;q^4)_n} \nonumber \\
&=-\frac{(-q;q^2)_\infty^2}{(q^4;q^8)_\infty} \sum_{n=0}^\infty \frac{q^{2n^2+2n}}{(q^4;q^4)_n(q^4;q^8)_{n+1}}, \label{new-exam8-3-T1} \\
T_2(q)&=\frac{(-q^3,-q^5,-q^{-1},-q;q^4)_\infty}{(q^2,-1;q^4)_\infty}\sum_{n=0}^\infty \frac{q^{2n^2}(-q^5,-q^3;q^4)_n}{(q^4,-q^3,-q^5,q^2,-q^4;q^4)_n} \nonumber \\
&=\frac{1}{2}q^{-1}\frac{(-q;q^2)_\infty^2}{(q^2,-q^4;q^4)_\infty}\sum_{n=0}^\infty \frac{q^{2n^2}}{(q^2;q^4)_n(q^8;q^8)_n}, \label{new-exam8-3-T2}\\
T_3(q)&=\frac{(q^3,q^5,q^{-1},q;q^4)_\infty}{(-1,-q^2;q^4)_\infty} \sum_{n=0}^\infty \frac{(-1)^nq^{2n^2}(q^5,q^3;q^4)_n}{(q^4,q^3,q^5,-q^4,-q^2;q^4)_n} \nonumber \\
&=-\frac{1}{2}q^{-1}\frac{(q;q^2)_\infty^2}{(-q^2;q^2)_\infty} \sum_{n=0}^\infty \frac{(-1)^nq^{2n^2}}{(-q^2;q^4)_n(q^8;q^8)_n}. \label{new-exam8-3-T3}
\end{align}
Substituting \eqref{Slater80} (with $q$ replaced by $q^4$) into \eqref{new-exam8-3-T1}, we obtain
\begin{align}
T_1(q)=-\frac{J_2^4J_8^2J_{56}^3}{J_1^2J_4^4J_{4,56}J_{16,56}J_{24,56}}.   \label{new-exam8-3-T1-result}
\end{align}
Substituting \eqref{Slater117} with $q$ replaced by $q^2$ (resp.\ $-q^2$) into \eqref{new-exam8-3-T2} (resp.\ \eqref{new-exam8-3-T3}), we obtain
\begin{align}
T_2(q)&=\frac{1}{2}q^{-1}\frac{J_2^2J_4J_{28}J_{6,56}J_{22,56}}{J_1^2J_8J_{56}J_{8,56}J_{24,56}}, \label{new-exam8-3-T2-result} \\
T_3(q)&=-\frac{1}{2}q^{-1}\frac{J_1^2J_2J_{14}J_{2,28}J_{10,28}}{J_2^2J_4J_{28}J_{4,28}J_{8,28}}.   \label{new-exam8-3-T3-result}
\end{align}
Substituting \eqref{new-exam8-3-T1-result}--\eqref{new-exam8-3-T3-result} into \eqref{exam8-3-F}, using the method in \cite{Garvan-Liang}, it is easy to verify that the second equality in \eqref{exam8-3} holds.
\end{proof}

\subsection{Example 9.}\label{sec-exam9}
The modular triples for this example are given in Table \ref{tab-9}.
\begin{table}[H]
\centering
\begin{tabular}{c|ccc}
  \hline
    \padedvphantom{I}{3ex}{3ex}
  $A$ &  \multicolumn{3}{c}{$\begin{pmatrix} 1 & -1/2 \\ -1/2 & 3/4 \end{pmatrix}$}  \\
  \hline
    \padedvphantom{I}{3ex}{3ex}
  $B$ & $\begin{pmatrix} -1/2 \\ 1/4 \end{pmatrix}$ & $\begin{pmatrix}  0 \\ 0 \end{pmatrix}$ & $ \begin{pmatrix} 0 \\ 1/2 \end{pmatrix}$ \\
  ~~$C$ & $1/28$ & $-3/56$ & $1/56$ \\
  \hline
\end{tabular}
\\[2mm]
\caption{Modular triples for Example 9}\label{tab-9}
\end{table}


\begin{theorem}\label{thm-9}
We have
\begin{align}
&\sum_{i,j\geq 0}\frac{q^{4i^2+3j^2-4ij-4i+2j}}{(q^8;q^8)_i(q^8;q^8)_j} \nonumber \\ &=2\frac{J_{16}J_{48,112}}{J_8^2}+q\frac{J_8J_{56}^2J_{112}J_{8,112}J_{24,112}}{J_4J_{4,112}J_{16,112}J_{28,112}J_{48,112}J_{52,112}}, \label{exam9-1} \\
&\sum_{i,j\geq 0}\frac{q^{4i^2+3j^2-4ij}}{(q^8;q^8)_i(q^8;q^8)_j} \nonumber \\
&=\frac{J_8J_{56}^2J_{112}J_{24,112}J_{40,112}}{J_4J_{12,112}J_{28,112}J_{32,112}J_{44,112}J_{48,112}}+2q^3\frac{J_{16}J_{32,112}}{J_8^2}, \label{exam9-2} \\
&\sum_{i,j\geq 0} \frac{q^{4i^2+3j^2-4ij+4j}}{(q^8;q^8)_i(q^8;q^8)_j}  \nonumber \\
&=\frac{J_8J_{56}^2J_{112}J_{8,112}J_{40,112}}{J_4J_{16,112}J_{20,112}J_{28,112}J_{32,112}J_{36,112}}
+2q^7\frac{J_{16}J_{16,112}}{J_8^2}. \label{exam9-3}
\end{align}
As a consequence, the Nahm sums $f_{A,B,C}(q^{56})$ for $(A,B,C)$ in Example 9 are modular functions on $\Gamma_1(3136)$.
\end{theorem}
\begin{proof}
(1) We need to find the 4-dissection:
\begin{align}\label{exam9-1-dissection}
\sum_{i,j\geq 0}\frac{q^{4i^2+3j^2-4ij-4i+2j}}{(q^8;q^8)_i(q^8;q^8)_j}=F_0(q^4)+qF_1(q^4)+q^2F_2(q^4)+q^3F_3(q^4),
\end{align}
where $F_0(q), F_1(q), F_2(q), F_3(q) \in \mathbb{Z}[\![q]\!]$.
Note that
$$4i^2+3j^2-4ij-4i+2j\equiv 1-(j-1)^2 \equiv \left\{\begin{array}{ll}
0 \pmod{4}, &j\equiv 0 \pmod{2},\\
1 \pmod{4}, & j\equiv 1\pmod{2}.
\end{array}\right.$$
Hence
\begin{align}\label{exam9-1-F3F4}
F_2(q)=F_3(q)=0.
\end{align}
We have
\begin{align}
F_0(q^4)&=\sum_{i,j\geq 0}\frac{q^{4i^2+3(2j)^2-8ij-4i+4j}}{(q^8;q^8)_i(q^8;q^8)_{2j}} =\sum_{i,j\geq 0}\frac{q^{4i^2-4i+12j^2-8ij+4j}}{(q^8;q^8)_i(q^8;q^8)_{2j}}.
\end{align}
Thus
\begin{align}
F_0(q)=\sum_{i,j\geq 0}\frac{q^{i^2-i+3j^2+j-2ij}}{(q^2;q^2)_i(q^2;q^2)_{2j}}.
\end{align}
We have
\begin{align}
F_0(q^{\frac{1}{2}})&=\sum_{i,j\geq 0}\frac{q^{(i^2-i)/2+(3j^2+j)/2-ij}}{(q;q)_i(q;q)_{2j}} =\sum_{j=0}^\infty \frac{q^{(3j^2+j)/2}}{(q;q)_{2j}}\sum_{i=0}^\infty \frac{q^{(i^2-i)/2}\cdot (q^{-j})^i}{(q;q)_i} \nonumber \\
&=\sum_{j=0}^\infty \frac{q^{(3j^2+j)/2}}{(q;q)_{2j}} (-q^{-j};q)_\infty =2(-q;q)_\infty \sum_{j=0}^\infty \frac{q^{j^2}(-q;q)_j}{(q;q)_{2j}}. \label{exam9-1-F0}
\end{align}
Substituting \eqref{Slater61} into \eqref{exam9-1-F0}, we deduce that
\begin{align}\label{exam9-1-F0-result}
F_0(q^{\frac{1}{2}})=2\frac{J_2J_{6,14}}{J_1^2}.
\end{align}

Next, we have
\begin{align}
qF_1(q^4)&=\sum_{i,j\geq 0}\frac{q^{4i^2-8ij+12j^2-8i+16j+5}}{(q^8;q^8)_i(q^8;q^8)_{2j+1}}.
\end{align}
Thus
\begin{align}
F_1(q)&=\sum_{i,j\geq 0}\frac{q^{i^2-2ij+3j^2-2i+4j+1}}{(q^2;q^2)_i(q^2;q^2)_{2j+1}}=\sum_{j=0}^\infty \frac{q^{3j^2+4j+1}}{(q^2;q^2)_{2j+1}}\sum_{i=0}^\infty \frac{q^{i^2-i}\cdot q^{-(2j+1)i}}{(q^2;q^2)_i} \nonumber \\
&=\sum_{j=0}^\infty \frac{q^{3j^2+4j+1}}{(q^2;q^2)_{2j+1}}(-q^{-2j-1};q^2)_\infty =(-q;q^2)_\infty \sum_{j=0}^\infty \frac{q^{2j^2+2j}(-q;q^2)_{j+1}}{(q^2;q^2)_{2j+1}}. \label{exam9-1-F1}
\end{align}
Substituting \eqref{Slater31} with $q$ replaced by $-q$ into \eqref{exam9-1-F1}, we deduce that
\begin{align}
F_1(q)=\frac{J_2J_{14}^2J_{28}J_{2,28}J_{6,28}}{J_1J_{1,28}J_{4,28}J_{7,28}J_{12,28}J_{13,28}}. \label{exam9-1-F1-result}
\end{align}
Substituting \eqref{exam9-1-F3F4}, \eqref{exam9-1-F0-result} and \eqref{exam9-1-F1-result} into \eqref{exam9-1-dissection}, we obtain \eqref{exam9-1}.

(2) We need to find the 4-dissection:
\begin{align}\label{exam9-2-dissection}
\sum_{i,j\geq 0}\frac{q^{4i^2+3j^2-4ij}}{(q^8;q^8)_i(q^8;q^8)_j}=G_0(q^4)+qG_1(q^4)+q^2G_2(q^4)+q^3G_3(q^4),
\end{align}
where $G_0(q), G_1(q), G_2(q), G_3(q) \in \mathbb{Z}[\![q]\!]$. Note that
\begin{align}
4i^2+3j^2-4ij\equiv -j^2  \equiv \left\{\begin{array}{ll}
0 \pmod{4}, &j\equiv 0 \pmod{2},\\
3 \pmod{4}, & j\equiv 1\pmod{2}.
\end{array}\right.
\end{align}
It follows that
\begin{align}\label{exam9-2-G1G2}
G_1(q)=G_2(q)=0.
\end{align}
We have
\begin{align}
G_0(q^4)=\sum_{i,j\geq 0}\frac{q^{4i^2+12j^2-8ij}}{(q^8;q^8)_i(q^8;q^8)_{2j}}.
\end{align}
Thus
\begin{align}
G_0(q)&=\sum_{i,j\geq 0}\frac{q^{i^2+3j^2-2ij}}{(q^2;q^2)_i(q^2;q^2)_{2j}} =\sum_{j=0}^\infty \frac{q^{3j^2}}{(q^2;q^2)_{2j}}\sum_{i=0}^\infty \frac{q^{i^2-i}\cdot q^{(1-2j)i}}{(q^2;q^2)_i} \nonumber \\
&=\sum_{j=0}^\infty \frac{q^{3j^2}}{(q^2;q^2)_{2j}}(-q^{1-2j};q^2)_\infty =(-q;q^2)_\infty \sum_{j=0}^\infty \frac{q^{2j^2}(-q;q^2)_j}{(q^2;q^2)_{2j}}. \label{exam9-2-G0}
\end{align}
Substituting \eqref{Slater33} (with $q$ replaced by $-q$) into \eqref{exam9-2-G0}, we deduce that
\begin{align}\label{exam9-2-G0-result}
G_0(q)=\frac{J_2J_{14}^2J_{28}J_{6,28}J_{10,28}}{J_1J_{3,28}J_{7,28}J_{8,28}J_{11,28}J_{12,28}}.
\end{align}

Next, we have
\begin{align}
q^3G_3(q^4)&=\sum_{i,j\geq 0}\frac{q^{4i^2+3(2j+1)^2-4i(2j+1)}}{(q^8;q^8)_i(q^8;q^8)_{2j+1}} =\sum_{i,j\geq 0} \frac{q^{4i^2-8ij+12j^2-4i+12j+3}}{(q^8;q^8)_i(q^8;q^8)_{2j+1}}.
\end{align}
We have
\begin{align}
G_3(q^{\frac{1}{2}})&=\sum_{i,j\geq 0}\frac{q^{(i^2-i)/2-ij+(3j^2+3j)/2}}{(q;q)_i(q;q)_{2j+1}} =\sum_{j=0}^\infty \frac{q^{(3j^2+3j)/2}}{(q;q)_{2j+1}}\sum_{i=0}^\infty  \frac{q^{(i^2-i)/2}\cdot q^{-ji}}{(q;q)_i} \nonumber \\
&=\sum_{j=0}^\infty \frac{q^{(3j^2+3j)/2}}{(q;q)_{2j+1}} (-q^{-j};q)_\infty =2(-q;q)_\infty \sum_{j=0}^\infty \frac{q^{j^2+j}(-q;q)_j}{(q;q)_{2j+1}}. \label{exam9-2-G3}
\end{align}
Substituting \eqref{Slater60} into \eqref{exam9-2-G3}, we deduce that
\begin{align}\label{exam9-2-G3-result}
G_3(q^{\frac{1}{2}})=2\frac{J_2J_{4,14}}{J_1^2}.
\end{align}
Substituting \eqref{exam9-2-G1G2}, \eqref{exam9-2-G0-result} and \eqref{exam9-2-G3-result} into \eqref{exam9-2-dissection}, we obtain \eqref{exam9-2}.

(3) We need to find the 4-dissection:
\begin{align}\label{exam9-3-dissection}
\sum_{i,j\geq 0}\frac{q^{4i^2+3j^2-4ij+4j}}{(q^8;q^8)_i(q^8;q^8)_j}=H_0(q^4)+qH_1(q^4)+q^2H_2(q^4)+q^3H_3(q^4),
\end{align}
where $H_0(q), H_1(q), H_2(q), H_3(q) \in \mathbb{Z}[\![q]\!]$. Note that
\begin{align}
4i^2+3j^2-4ij+4j\equiv -j^2  \equiv \left\{\begin{array}{ll}
0 \pmod{4}, &j\equiv 0 \pmod{2},\\
3 \pmod{4}, & j\equiv 1\pmod{2}.
\end{array}\right.
\end{align}
It follows that
\begin{align}\label{exam9-3-H1H2}
H_1(q)=H_2(q)=0.
\end{align}
We have
\begin{align}
H_0(q^4)=\sum_{i,j\geq 0}\frac{q^{4i^2+12j^2-8ij+8j}}{(q^8;q^8)_i(q^8;q^8)_{2j}}.
\end{align}
Thus
\begin{align}
H_0(q)&=\sum_{i,j\geq 0}\frac{q^{i^2+3j^2-2ij+2j}}{(q^2;q^2)_i(q^2;q^2)_{2j}} =\sum_{j=0}^\infty \frac{q^{3j^2+2j}}{(q^2;q^2)_{2j}} \sum_{i=0}^\infty \frac{q^{i^2-i}\cdot q^{(1-2j)i}}{(q^2;q^2)_i} \nonumber \\
&=\sum_{j=0}^\infty \frac{q^{3j^2+2j}}{(q^2;q^2)_{2j}}(-q^{1-2j};q^2)_\infty =(-q;q^2)_\infty \sum_{j=0}^\infty \frac{q^{2j^2+2j}(-q;q^2)_j}{(q^2;q^2)_{2j}}. \label{exam9-3-H0}
\end{align}
Substituting \eqref{Slater32} (with $q$ replaced by $-q$) into \eqref{exam9-3-H0}, we deduce that
\begin{align}\label{exam9-3-H0-result}
H_0(q)=\frac{J_2J_{14}^2J_{28}J_{2,28}J_{10,28}}{J_1J_{4,28}J_{5,28}J_{7,28}J_{8,28}J_{9,28}}.
\end{align}

Next, since
\begin{align}
q^3H_3(q^4)=\sum_{i,j\geq 0}\frac{q^{4i^2-8ij+12j^2-4i+20j+7}}{(q^8;q^8)_i(q^8;q^8)_{2j+1}},
\end{align}
we have
\begin{align}
H_3(q^{\frac{1}{2}})&=q^{\frac{1}{2}}\sum_{i,j\geq 0}\frac{q^{(i^2-i)/2-ij+(3j^2+5j)/2}}{(q;q)_i(q;q)_{2j+1}} =q^{\frac{1}{2}}\sum_{j=0}^\infty \frac{q^{(3j^2+5j)/2}}{(q;q)_{2j+1}}\sum_{i=0}^\infty \frac{q^{(i^2-i)/2}\cdot q^{-ji}}{(q;q)_i} \nonumber \\
&=q^{\frac{1}{2}}\sum_{j=0}^\infty \frac{q^{(3j^2+5j)/2}}{(q;q)_{2j+1}}(-q^{-j};q)_\infty =2q^{\frac{1}{2}}(-q;q)_\infty \sum_{j=0}^\infty \frac{q^{j^2+2j}(-q;q)_j}{(q;q)_{2j+1}}. \label{exam9-3-H3}
\end{align}
Substituting \eqref{Slater59} into \eqref{exam9-3-H3}, we deduce that
\begin{align}\label{exam9-3-H3-result}
H_3(q^{\frac{1}{2}})=2q^{\frac{1}{2}}\frac{J_2J_{2,14}}{J_1^2}.
\end{align}
Substituting \eqref{exam9-3-H1H2}, \eqref{exam9-3-H0-result} and \eqref{exam9-3-H3-result} into \eqref{exam9-3-dissection}, we obtain \eqref{exam9-3}.
\end{proof}

\subsection{Example 10.}\label{sec-exam10}
The modular triples for this example are given in Table \ref{tab-10}.
\begin{table}[H]
\centering
\begin{tabular}{c|ccc}
  \hline
    \padedvphantom{I}{3ex}{3ex}
  $A$ &  \multicolumn{3}{c}{$\begin{pmatrix} 4/3 & 2/3 \\ 2/3 & 4/3 \end{pmatrix}$}  \\
  \hline
\padedvphantom{I}{3ex}{3ex}
$B$ & $\begin{pmatrix} -2/3 \\ -1/3 \end{pmatrix}$ & $\begin{pmatrix}  -1/3 \\ -2/3 \end{pmatrix}$ & $\begin{pmatrix} 0 \\ 0 \end{pmatrix}$ \\
  ~~ $C$ & $1/30$ & $1/30$ & $-1/30$ \\
  \hline
\end{tabular}
\\[2mm]
\caption{Modular triples for Example 10}\label{tab-10}
\end{table}


Since the $(1,1)$-entry and the $(2,2)$-entry of $A$ are the same, the first and the second vectors give the same identity.

This example has been discussed by Vlasenko and Zwegers \cite[p.\ 633, Table 1]{VZ}. They found the following conjectural identities.
\begin{conj}\label{conj-10}
We have
\begin{align}
&\sum_{i,j\geq 0} \frac{q^{2i^2+2ij+2j^2-2i-j}}{(q^3;q^3)_i(q^3;q^3)_j} \nonumber \\
&=\frac{1}{(q^3;q^3)_\infty} \sum_{n=-\infty}^\infty
(-1)^n\left(2q^{\frac{45}{2}n^2+\frac{9}{2}n}+q^{\frac{45}{2}n^2+\frac{39}{2}n+4}-q^{\frac{45n^2}{2}+\frac{69}{2}n+13}   \right)  \nonumber \\
&=\frac{1}{J_3}\left(2J_{18,45}+qJ_{12,45}+q^4J_{3,45}\right), \label{conj-10-1} \\
&\sum_{i,j\geq 0} \frac{q^{2i^2+2ij+2j^2}}{(q^3;q^3)_i(q^3;q^3)_j} \nonumber \\
&=\frac{1}{(q^3;q^3)_\infty} \sum_{n=-\infty}^\infty \left( 2q^{\frac{45}{2}n^2+\frac{27}{2}n+2}+q^{\frac{45}{2}n^2+\frac{3}{2}n}-q^{\frac{45}{2}n^2+\frac{33}{2}n+3}\right) \nonumber \\
&=\frac{1}{J_3}\left(J_{21,45}-q^3J_{6,45}+2q^2J_{9,45}  \right). \label{conj-10-2}
\end{align}
\end{conj}
For each of the identity, the first equality is the one given by \cite{VZ}, while the second equality follows from the first one and \eqref{Jacobi}. As a consequence, the Nahm sums $f_{A,B,C}(q^{30})$ for $(A,B,C)$ in Example 10 are modular functions on $\Gamma_1(900)$.

Cherednik and Feigin \cite{Feigin} confirmed the modularity of the Nahm sums in this example via the nilpotent double affine Hecke algebras. They \cite[p.\ 1074]{Feigin} proved that the sum sides of these identities are modular functions, and then said that to prove \eqref{conj-10-1} and \eqref{conj-10-2}, ``one needs to compare only few terms in the $q$-expansions to establish their coincidence''. However, it seems difficult to prove Conjecture \ref{conj-10} in this way since it is hard to determine the number of terms need to be compared.

Though we cannot find purely $q$-series proofs for this theorem (which would be very exciting), we find equivalent formulas for it.
\begin{conj}\label{conj-10-Wang}
We have
\begin{align}
\sum_{i,j\geq 0} \frac{q^{2i^2+2ij+2j^2-2i-j}}{(q^3;q^3)_i(q^3;q^3)_j}&=3\frac{J_{18,45}}{J_3}-\frac{J_{5}J_{1,15}J_{4,15}}{J_{15}^2J_{3,15}},   \label{conj-10-1-Wang} \\
\sum_{i,j\geq 0} \frac{q^{2i^2+2ij+2j^2}}{(q^3;q^3)_i(q^3;q^3)_j} &=\frac{J_5J_{2,15}J_{7,15}}{J_{15}^2J_{6,15}}+3q^2\frac{J_{9,45}}{J_3}. \label{conj-10-2-Wang}
\end{align}
\end{conj}
The equivalence of Conjectures \ref{conj-10} and \ref{conj-10-Wang} can be proved easily using the method in \cite{Garvan-Liang}.

\subsection{Example 11.}\label{sec-exam11}
The modular triples for this example are given in Table \ref{tab-11}.
\begin{table}[H]
\centering
\begin{tabular}{c|ccc}
  \hline
  \padedvphantom{I}{3ex}{3ex}
  $A$ &  \multicolumn{3}{c}{$\begin{pmatrix} 1 &-1/2 \\ -1/2 & 1 \end{pmatrix}$}   \\
  \hline
  \padedvphantom{I}{3ex}{3ex}
  $B$ & $\begin{pmatrix} -1/2 \\ 0 \end{pmatrix}$ & $\begin{pmatrix} 0 \\ -1/2 \end{pmatrix}$ & $\begin{pmatrix} 0 \\ 0 \end{pmatrix}$ \\
  ~~$C$ & $1/20$ & $1/20$ & $-1/20$ \\
  \hline
\end{tabular}
\\[2mm]
\caption{Modular triples for Example 11}\label{tab-11}
\end{table}


This example has been discussed by Vlasenko and Zwegers \cite[p.\ 629, Theorem 3.2]{VZ}. They gave modular function representations for it.
We shall state their result in the following equivalent form and describe the modularity.
\begin{theorem}\label{thm-11}
We have
\begin{align}
\sum_{i,j\geq 0}\frac{q^{i^2+j^2-ij-i}}{(q^2;q^2)_i(q^2;q^2)_j} &=2\frac{J_4J_{8,20}}{J_2^2}+\frac{J_2^2J_{10}J_{20}^3}{J_1J_4J_{1,20}J_{5,20}J_{8,20}J_{9,20}},  \label{exam11-1} \\
\sum_{i,j\geq 0}\frac{q^{i^2+j^2-ij}}{(q^2;q^2)_i(q^2;q^2)_j}&=\frac{J_2^2J_{10}J_{20}^3}{J_1J_4J_{3,20}J_{4,20}J_{5,20}J_{7,20}}+2q\frac{J_4J_{4,20}}{J_2^2}. \label{exam11-2}
\end{align}
As a consequence, the Nahm sums $f_{A,B,C}(q^{20})$ for $(A,B,C)$ in Example 11 are modular functions on $\Gamma_1(400)$.
\end{theorem}
We will not repeat the proof in \cite{VZ}. But for the convenience of the reader, we sketch briefly the steps. To prove \eqref{exam11-1}, one may sum over $i$ using \eqref{Euler}, and then split the sum into two sums corresponding to even and odd values of $j$, respectively. Then for the sum with $j$ even, we can use \eqref{Slater46}. For the sum with $j$ odd, we can use \eqref{Slater97}.

Similarly, the identity \eqref{exam11-2} can be reduced to the identities \eqref{Slater19} and \eqref{Slater44}.

So far we have discussed all the eleven rank two examples discovered by Zagier. Before closing this paper, we make two remarks. First, according to the notion used in \cite{Wang}, all of the above identities for Zagier's examples are Rogers--Ramanujan type identities of index $(1,1)$, except that Theorem \ref{thm-2} also contains double sums of index $(1,2)$. Here we slightly modified the notion in \cite{Wang} by allowing the product side to be finite sums of infinite products instead of just one single product. Second,  Zagier \cite{Zagier} also stated many possible modular triples when the rank $r=3$. We have verified all of them by proving some identities of index $(1,1,1)$. This will be discussed in detail in a separate paper.

\begin{rem}\label{rem-CRW}
After the first version of this paper (arXiv:2210.10748v1) has been written, Conjectures \ref{conj-exam5} and \ref{conj-10} have been confirmed by Cao, Rosengren and the author \cite{CRW2024} via $q$-series approach. Therefore, the modularity of all of Zagier's rank two examples in \cite[Table 2]{Zagier} have now been confirmed.
\end{rem}


\subsection*{Acknowledgements}
This work was supported by the National Natural Science Foundation of China (12171375). We thank Prof.\ Haowu Wang for some helpful comments in Remarks \ref{rem-exam2} and \ref{rem-exam4}. We are also grateful to the referees for carefully reading the manuscript and providing valuable comments and suggestions.


\end{document}